\documentclass[12pt]{article}
\usepackage[utf8]{inputenc}
\usepackage[english]{babel}
\usepackage{amsmath,amsthm,amsfonts,amssymb,mathtools}
\usepackage{enumerate,graphicx,xcolor}
\usepackage{lmodern}
\usepackage[T1]{fontenc}
\usepackage{comment}
\usepackage{subcaption}
\usepackage{tikz}
\usepackage{macros}
\usepackage{dsfont}
\usepackage[left=2cm,top=2.5cm,right=2cm,bottom=2.5cm]{geometry}
\usepackage[activate={true,nocompatibility},final,tracking=true,kerning=true,spacing=true,factor=1100,stretch=10,shrink=10]{microtype}
\usepackage{tikz}
\usetikzlibrary{calc}

\usepackage{xcolor}
\usepackage{mdframed}
\usepackage[shortlabels]{enumitem}
\usepackage{indentfirst}
\usepackage{hyperref}
\usepackage{stackrel}
\usepackage[dvipsnames,svgnames,table]{xcolor}
\usepackage{colortbl}

\title{Combinatorial formulas for the multivariate Markov-Krein transform}
\author{ Julián Zazueta-Obeso \\ 
julian.zazueta@cimat.mx\\ 
CIMAT, Guanajuato}
\date{ }

\begin{document}

\maketitle
\begin{abstract}
In this work, we study the Markov--Krein correspondence in both the univariate and multivariate settings. Our main result provides an explicit moment--moment formula expressing the multivariate moments of the inverse Markov--Krein transform in terms of the multivariate moments of the Markov-Krein transform via cyclic interval partitions. We also derive moment--cumulant relations involving classical, Boolean and free cumulants.\\

\end{abstract}

\section{Introduction}
The Markov--Krein correspondence provides a nonlinear relation between signed measures on the real line and probability measures through analytic transforms. Given a Rayleigh measure $\tau$ on $\mathbb{R}$, its Markov--Krein transform $\mu$ is characterized by the identity
\begin{equation}\label{eq:MK-def}
\int_{\mathbb{R}} \frac{\mu(du)}{z-u}
=
\exp\left(
-\int_{\mathbb{R}} \log(z-u)\,\tau(du)
\right),
\qquad \Im z>0,
\end{equation}
or, equivalently, by the analytic relation between the corresponding Cauchy transforms,
\begin{equation}\label{eq:MK-cauchy}
G_\mu'(z)+G_\mu(z)G_\tau(z)=0.
\end{equation}
This correspondence has been systematically studied in several mathematical contexts. In particular, it arises from the fundamental contributions in the representation of the moment problem \cite{kreuin1977markov} and appears in exponential representations of functions with negative imaginary part \cite{OnExponential}, in spectral shift functions of self-adjoint operators \cite{harary2014graphical,flajolet1980combinatorial}, in the study of Young diagrams \cite{yongdiagrams-kerov} and in the context of interlacing measures \cite{interlacingmeasures}.

In the univariate setting, the analytic relation that defines the Markov--Krein correspondence translates into explicit identities between the moments of $\mu$ and those of $\tau$
\begin{equation}\label{eq:momentos univariados}
    \sum_{n=0}^\infty h_nz^n=\exp\left(\sum_{n=1}^\infty\frac{p_n}{n}z^n\right),
\end{equation}
where $h_n=\int_{-\infty}^\infty t^n\mu(dt)$ and $p_n=\int_{-\infty}^\infty t^n\tau(dt)$ are the $n$-th moments of $\mu$ and $\tau$, respectively.

These moment relations admit a combinatorial description and are closely connected to cumulant structures arising in non-commutative probability, see \cite{arizmendi,arizmendi2023finite,arizmendi2016free,katsunori}. In particular, the inverse Markov--Krein correspondence admits moment formulas expressed in terms of free cumulants 
\begin{equation}\label{p_n-r_n univariate}
    p_n=\sum_{\pi\in NC(n)}(n-1-\#(\pi))r_\pi^{\mu},
\end{equation}
where $r_n^\mu$ represents the $n$-th free cumulant corresponding to $\mu$, formula that was obtained in the works of T. Hasebe and K. Fujie \cite{katsunori} related to rotationally invariant Hermitian random matrices.

 Moreover, these moments can be interpreted from the perspective of operator theory as follows. If we consider \(A\) as a bounded self-adjoint operator on a separable Hilbert space \(\mathcal{H}\), and \(\xi \in \mathcal{H}\) a unit vector, then we may define the \(n\)-th moment of the Markov--Krein transform \(h_n\), and the \(n\)-th moment of the inverse Markov--Krein transform \(p_n\), as follows
\begin{equation}\label{h_n and p_n in operator theory}
h_n = \langle A^n \xi, \xi \rangle \qquad \text{and} \qquad p_n = \operatorname{Tr}\big(A^n - (q Aq)^n \big), \qquad \text{for all } n \in \mathbb{N},
\end{equation}
where $q:\mathcal{H}\longrightarrow\mathcal{H}$ denotes the orthogonal projection onto the hyperplane \(\mathcal{H}_\xi\) in \(\mathcal{H}\) orthogonal to the vector \(\xi\), which is defined by
    \begin{equation*}
        q(\eta)=\eta-p(\eta)=\eta-\left<\eta,\xi\right>\xi,\qquad \text{for all }\eta\in \mathcal{H}.
    \end{equation*}
These formulas originate primarily in the works of I. M. Lifshits \cite{lifshits1952problem}, were further developed by M. G. Krein \cite{Krein-teoria-operadores}.

Recent developments in non-commutative probability, motivated by relation \eqref{h_n and p_n in operator theory}, have led to the introduction of a multivariate version of the Markov--Krein correspondence. This multivariate framework appears in two parallel works developed in different probabilistic notions, namely in the theory of free independence of type $B$ prime \cite{katsunori2023} and in the theory of cyclic conditional independence \cite{arizmendi}. More precisely, if $\mathcal{A}\subset \mathcal{B}(\mathcal{H})$ is a
\(\ast\)-subalgebra, the functional
\(h_n:\mathcal A^n\to\mathbb C\) defined by
\begin{equation}\label{ec-2}
h_n(A_1,\ldots,A_n)
=
\langle A_1\cdots A_n\xi,\xi\rangle,
\end{equation}
is called the \(n\)-th multivariate moment of the Markov--Krein transform of $h$,
since it extends the univariate formula in
\eqref{h_n and p_n in operator theory}. Similarly, the functional
\(p_n:\mathcal A^n\to\mathbb C\), defined by
\begin{equation}\label{ec-1}
    p_n(A_1,\ldots,A_n)= \mathrm{Tr}(A_1\cdots A_n-q A_1 q\cdots q A_n q),
\end{equation}
is called the \(n\)-th multivariate moment of the inverse Markov--Krein transform of $h$. With this multivariate extension, the authors of the independent works \cite{katsunori2023} and \cite{arizmendi} generalized formula \eqref{p_n-r_n univariate} to the multivariate setting.

The main result of this paper establishes that multivariate moments $p_n$ in \eqref{ec-1} can be expressed in terms of moments $h_n$ in \eqref{ec-2} through sums indexed by cyclic interval partitions,
\begin{equation}\label{eq:cyclic-partitions}
p_n(A_1,\ldots,A_n)
=
\sum_{\pi\in \mathcal{CI}(n)}(-1)^{\#(\pi)-1} h_\pi[A_1,\ldots,A_n],
\end{equation}
where $\mathcal{CI}(n)$ denotes the set of cyclic intervals partitions of $\{1,2,\ldots,n\}$.

The main purpose of this paper is to study these moment relations \eqref{eq:momentos univariados} and \eqref{eq:cyclic-partitions} in both univariate and multivariate settings, respectively. We establish explicit moment--moment formulas for the Markov--Krein correspondence and derive expressions for the moments of its inverse $p_n$ in terms of the different types of cumulants, classical, boolean and free, corresponding to the Markov--Krein transform of $h$, obtaining the multivariate version of \eqref{p_n-r_n univariate}. 

Another contribution of this work is to derive alternative analytic expressions for the equation \eqref{eq:MK-cauchy} in the univariate case, in order to obtain alternative representations and investigate whether such relations can be extended to the multivariate setting. However, this remains an open problem, as it has only been possible to characterize the Markov--Krein correspondence in the multivariate case through combinatorial relations.

Finally, we present some remarks of surprising connections with related works. In particular, we relate our results to the recent work of Mingo and Tseng \cite{mingo2024infinitesimal}, highlighting structural similarities between the Markov--Krein correspondence and the Boolean-like cumulants introduced in the context of infinitesimal freeness. Moreover, we establish a connection with finite free probability and infinitesimal distributions, as developed in the work of Arizmendi, Garza and Perales \cite{arizmendi2023finite}.

The paper is organized as follows. In Section 2, we introduce the main tools used throughout the paper, including basic notions of non-commutative probability, partitions, and cumulants. In Section 3, we present the Markov-Krein correspondence, describing its analytic and combinatorial characterizations, as well as its main properties. Finally, in the last section, we present the main results of this paper, including the extension of the Markov--Krein transform to the multivariate setting, together with its important properties and connections, the moment--moment formula, and the moment--cumulant formulas.

\section{Preliminaries}
In this section, we introduce the framework of non-commutative probability, together with the combinatorial notions of partitions and cumulants that will be used throughout the paper. The main references are \cite{book:nica,mingo2017free}.
\subsection{Non-commutative probability}
\begin{definition}
    A \textbf{non-commutative probability space}, \((\mathcal{A},\varphi)\), consists of a unital algebra \(\mathcal{A}\) over \(\mathbb{C}\) and a unital linear functional
    \begin{equation*}
        \varphi : \mathcal{A}\rightarrow \mathbb{C},\quad\text{such that}\quad \varphi(1_{\mathcal{A}} )=1.
    \end{equation*}
\end{definition}
The elements \(a\in\mathcal{A}\) are called non-commutative random variables in \((\mathcal{A},\varphi)\). For simplicity, we will simply refer to them as random variables. We say that the non-commutative probability space \((\mathcal{A},\varphi)\) is \textbf{tracial} (or just $\varphi$ is tracial) if it satisfies
\begin{equation*}
    \varphi(ab)=\varphi(ba),\qquad \forall a,b\in\mathcal{A}.
\end{equation*}
A functional \(       f_n:\mathcal A^n\to\mathbb C
    \)
    is said to be \textbf{cyclically invariant} if
    \[
        f_n(a_1,\ldots,a_n)
        =
        f_n(a_2,\ldots,a_n,a_1),
        \qquad \forall \ a_1,\ldots,a_n\in\mathcal A.
    \]
\subsection{Partitions}
In this subsection we present the different families of partitions that will be used throughout this article.
\begin{definition}
Let $n\in\mathbb{N}$. We denote by $[n]$ the totally ordered set $\{1<2<\cdots<n\}$.
\begin{itemize}
    \item [(1) ] We say that $\pi=\{V_1,\ldots,V_r\}$ is a \textbf{partition} of $[n]$ if and only if $V_i$, with $1\leq i\leq r$, are non-empty subsets of $[n]$, are pairwise disjoint, and satisfy $V_1\cup V_2\cup\cdots \cup V_r=[n]$. We say that $V_1,\ldots,V_r$ are the \textbf{blocks} of $\pi$. We denote by $\#(\pi)$ the number of blocks of $\pi$. The set of all partitions of $[n]$ is denoted by $\mathcal{P}(n)$.
    
    \item [(2) ] Every partition defines an equivalence relation on $[n]$ and vice versa. Given $p,q\in [n]$ and a partition $\pi\in\mathcal{P}(n)$, we write $p\sim_\pi q$ if and only if $p$ and $q$ belong to the same block of $\pi$.
    
    \item [(3) ] We say that a partition $\pi\in\mathcal{P}(n)$ is a \textbf{crossing partition} if there exist $1\leq p_1<q_1<p_2<q_2\leq n$ such that $p_1\sim_\pi p_2$, $q_1\sim_\pi q_2$, and $p_1\nsim_\pi q_1$. Otherwise, we say that $\pi$ is a \textbf{non-crossing partition}. The set of non-crossing partitions of $[n]$ is denoted by $\mathcal{NC}(n)$.
    
    \item [(4) ] We say that a block $V$ is an \textbf{interval} if $V$ is of the form $V=\{k,k+1,\ldots,k+l\}$ for some $k\geq 1$ and $0\leq l\leq n-k$.
    
    \item [(5) ] We say that $\pi\in\mathcal{P}(n)$ is an \textbf{interval partition} if every block of $\pi$ is an interval. The set of interval partitions of $[n]$ is denoted by $\mathcal{I}(n)$.
    
    \item [(6) ] Let $\pi,\sigma\in\mathcal{P}(n)$. We say that $\pi\leq \sigma$ if every block of $\pi$ is contained in some block of $\sigma$. This partial order is called the \textbf{reverse refinement order}.
    
    \item [(7) ] Let $\pi\in\mathcal{P}(n)$. We say that $\overline{\pi}$ is the \textbf{non-crossing closure} of $\pi$ if it is the smallest non-crossing partition such that $\pi\leq \overline{\pi}$.
    
    \item [(8) ] Let $\pi\in\mathcal{P}(n)$. We say that $\pi$ is \textbf{connected} if its non-crossing closure satisfies $\overline{\pi}=1_n$, where $1_n$ is the partition consisting of a single block $1_n=\{1,2,\ldots,n\}$. The set of connected partitions is denoted by $\mathcal{P}_{conn}(n)$.
    
    \item [(9) ] Let $\pi\in\mathcal{P}(n)$. We say that $\widehat{\pi}$ is the \textbf{interval closure} of $\pi$ if it is the smallest interval partition such that $\pi\leq \widehat{\pi}$.
    
    \item [(10) ] Let $\pi\in\mathcal{P}(n)$. We say that $\pi$ is \textbf{irreducible} if its interval closure is the largest partition, that is, $\widehat{\pi}=1_n$. If $\pi\in\mathcal{NC}(n)$, saying that $\pi$ is irreducible is equivalent to the condition $1\sim_\pi n$. The sets of irreducible partitions and irreducible non-crossing partitions are denoted by $\mathcal{P}_{irr}(n)$ and $\mathcal{NC}_{irr}(n)$, respectively.
\end{itemize}
\end{definition}

\begin{definition}\label{def.kreweras}
The \textbf{Kreweras complement} is a map $Kr:\mathcal{NC}(n)\rightarrow\mathcal{NC}(n)$ defined as follows. Consider additional symbols $\overline{1},\overline{2},\ldots,\overline{n}$ and interlace them with the numbers $1,2,\ldots,n$ in the following way:
\begin{equation*}
1<\overline{1}<2<\overline{2}<\cdots<\overline{n-1}<n<\overline{n}.
\end{equation*}
For $\pi\in\mathcal{NC}(n)$, we define $Kr(\pi)\in\mathcal{NC}(\overline{1},\ldots,\overline{n})\cong \mathcal{NC}(n)$ as the largest element (with respect to the reverse refinement order) such that $\pi\cup Kr(\pi)\in\mathcal{NC}(1,\overline{1},\ldots,n,\overline{n})$.
\end{definition}

We introduce the following notation: for each \(j=2,3,\ldots,n\), let
\[
1_{n,j}:=(j,j+1,\ldots,n,1,\ldots,j-1),
\]
where the indices are understood modulo \(n\). We refer to
\(1_{n,j}\) as the cyclic ordering of \([n]\) starting at \(j\). Since \(1_{n,j}\) is regarded as an ordered tuple, different starting points determine different cyclic orderings. In particular,
\(
1_{n,j}\neq 1_{n,k},
\) whenever $j\neq k$.
\begin{definition}
Let $n\in \mathbb{N}$. An element \(\pi\) is called a \textbf{cyclic interval partition} if one of the following conditions holds:
\begin{itemize}
    \item[(1)] \(\pi =\{1_{n,j}\}\) for some \(j = 2,3,\ldots,n\);
    \item[(2)] \(\pi\in\mathcal{P}(n)\) and every block of \(\pi\) is either an interval or the complement of an interval with respect to the cyclic order on \([n]\).
\end{itemize}
We denote by \(\mathcal{CI}(n)\) the set of cyclic interval partitions of \([n]\). Equivalently,
\[
\mathcal{CI}(n)
=\left\{
\{1_{n,j} \}: j = 2,3,\ldots,n
\right\}\cup
\left\{
\pi : \pi \text{ satisfies (2)}
\right\}.
\]
\end{definition}

\begin{example}\label{Particiones_por_intervalos_ciclicas}
Some examples of cyclic interval partitions of $[n]$ are the following:

\begin{itemize}

\item[(1)] For $n=2$, the cyclic interval partitions are
\[
\pi_1 = \{\{1,2\}\},
\qquad
\pi_2 = \{(2,1)\},
\qquad
\pi_3 = \{\{1\},\{2\}\}.
\]

\item[(2)] For $n=3$, the cyclic interval partitions are
\begin{align*}
\pi_1 &= \{\{1,2,3\}\},
\qquad
\pi_2 = \{(2,3,1)\},
\qquad
\pi_3 = \{(3,1,2)\},
\\
\pi_4 &= \{\{1,2\},\{3\}\},
\qquad
\pi_5 = \{\{2,3\},\{1\}\},
\qquad
\pi_6 = \{\{3,1\},\{2\}\},
\\
\pi_7 &= \{\{1\},\{2\},\{3\}\}.
\end{align*}

\item[(3)] For $n=4$, the cyclic interval partitions are
\begin{align*}
&\pi_1 = \{\{1,2,3,4\}\},
\qquad
\pi_2 = \{(2,3,4,1)\},
\qquad
\pi_3 = \{(3,4,1,2)\},
\\
&\pi_4 = \{(4,1,2,3)\},,\qquad \pi_5 = \{\{1,2,3\},\{4\}\},
\qquad
\pi_6 = \{\{2,3,4\},\{1\}\},
\\
&\pi_7 = \{\{3,4,1\},\{2\}\},
\qquad
\pi_8 = \{\{4,1,2\},\{3\}\},\qquad \pi_9 = \{\{1,2\},\{3,4\}\},\\
&\pi_{10} = \{\{2,3\},\{4,1\}\},\qquad \pi_{11} = \{\{1,2\},\{3\},\{4\}\},
\qquad
\pi_{12} = \{\{2,3\},\{4\},\{1\}\}
\\
&\pi_{13} = \{\{3,4\},\{1\},\{2\}\},
\qquad
\pi_{14} = \{\{4,1\},\{2\},\{3\}\},\qquad \pi_{15} = \{\{1\},\{2\},\{3\},\{4\}\}.
\end{align*}
\end{itemize}
\end{example}
\begin{remark}
    There is a natural bijection between cyclic interval partitions and
configurations of cuts on a circle. More precisely, place the points \(1,2,\ldots,n\) on a circle according to
their cyclic order. By selecting a subset of edges between consecutive points
and making cuts at those locations, the connected components formed between
consecutive cuts determine the blocks of a cyclic interval partition.
Conversely, every cyclic interval partition can be uniquely obtained from such
a configuration of cuts. For example,
\[
\pi=\{\{1,8,9\},\{2,3,4\},\{5,6\},\{7\}\}\in\mathcal{CI}(9)
\]
corresponds to a configuration with cuts between the blocks
\(\{1,8,9\}\), \(\{2,3,4\}\), \(\{5,6\}\), and \(\{7\}\) on the circle, see Figure \ref{fig:particiones_circulo}. 

\begin{figure}[ht]
    \centering
    \begin{tikzpicture}

        \begin{scope}[shift={(0,0)}]
            \def\r{2}
            \def\L{1}
            \def\R{2.5} 

            \draw[thick] (0,0) circle (\r);

            \foreach \k in {1,...,9}{
                \pgfmathsetmacro{\ang}{90 - (\k-1)*40}

                \coordinate (P\k) at ({\r*cos(\ang)},{\r*sin(\ang)});
                \coordinate (Q\k) at ({\R*cos(\ang)},{\R*sin(\ang)});

                \fill (P\k) circle (1.2pt);
                \node at ({0.82*\r*cos(\ang)},{0.82*\r*sin(\ang)}) {\k};
                \draw[blue, thick] (P\k) -- (Q\k);
            }

            \draw[blue, thick] (Q8) arc[start angle=170, end angle=90, radius=\R];

            \draw[blue, thick] (Q2) arc[start angle=50, end angle=-30, radius=\R];

            \draw[blue, thick] (Q5) arc[start angle=-70, end angle=-110, radius=\R];
        \end{scope}

        \begin{scope}[shift={(8,0)}]
            \def\r{2}
            \def\Lin{0.7}   
            \def\Lout{1}    
            \def\Rin{1.3}   
            \def\Rout{3}    

            \draw[thick] (0,0) circle (\r);

            \foreach \k in {1,...,9}{
                \pgfmathsetmacro{\ang}{90 - (\k-1)*40}

                \coordinate (P\k) at ({\r*cos(\ang)},{\r*sin(\ang)});

                \fill (P\k) circle (1.2pt);
                \node at ({0.82*\r*cos(\ang)},{0.82*\r*sin(\ang)}) {\k};
            }

            \draw[orange, thick] ({\Rin*cos(70)},{\Rin*sin(70)}) -- ({\Rout*cos(70)},{\Rout*sin(70)});

            \draw[orange, thick] ({\Rin*cos(-50)},{\Rin*sin(-50)}) -- ({\Rout*cos(-50)},{\Rout*sin(-50)});

            \draw[orange, thick] ({\Rin*cos(-130)},{\Rin*sin(-130)}) -- ({\Rout*cos(-130)},{\Rout*sin(-130)});

            \draw[orange, thick] ({\Rin*cos(-170)},{\Rin*sin(-170)}) -- ({\Rout*cos(-170)},{\Rout*sin(-170)});
        \end{scope}

    \end{tikzpicture}
    \caption{Partition $\pi\in\mathcal{CI}(9)$ such that 
$\pi=\{\{1,8,9\},\{2,3,4\},\{5,6\},\{7\}\}$.
On the left, we depict the partition $\pi$ by grouping elements in the same block. On the right, we show its representation by cuts on the circle.}
    \label{fig:particiones_circulo}
\end{figure}
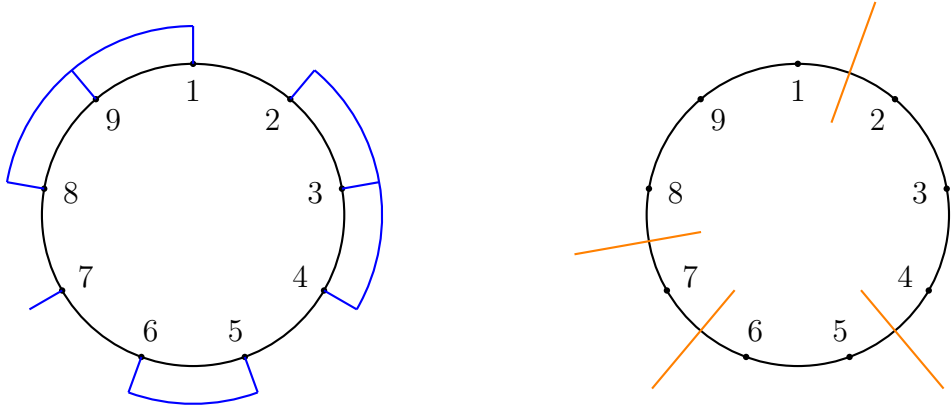
\end{remark}
\begin{lemma}
The cardinality of the set $\mathcal{CI}(n)$ is
\[
|\mathcal{CI}(n)| = 2^n - 1.
\]
\end{lemma}

\begin{proof}
Since cyclic interval partitions can be identified with configurations of cuts on the circle, counting the number of such partitions is equivalent to counting the number of ways to place cuts between consecutive points on the circle. There are $n$ possible locations where a cut can be placed, corresponding to the $n$ gaps between consecutive points in the cyclic order. A cyclic interval partition is obtained by choosing any nonempty subset of these $n$ possible cuts. Thus, the total number of cyclic interval partitions is
\[
|\mathcal{CI}(n)|
=
\sum_{k=1}^{n}
\binom{n}{k}
=
2^n - 1.
\]
\end{proof}

\begin{remark}\label{remark-cyclic-partitions}
Let $\widehat{\mathcal{CI}}(n)$ denote the set of cyclic interval partitions obtained by removing the elements $\{1_{n,2}\},\{1_{n,3}\},\ldots,\{1_{n,n}\}$ from $\mathcal{CI}(n)$, that is,
\[
\widehat{\mathcal{CI}}(n) = \mathcal{CI}(n) \setminus \{\{1_{n,2}\},\{1_{n,3}\},\ldots,\{1_{n,n}\}\}.
\]
This family coincides with the class of cyclic interval partitions considered in several works \cite{arizmendi2023cyclic,diaz2022combinatorics,mingo2024infinitesimal}. The main difference is that, in our convention, we treat the block of length $n$ as an ordered interval, and hence there are $n$ possible starting positions. Moreover, in their context, the cardinality of $\widehat{\mathcal{CI}}(n)$ is $2^{n}-n$. Also,  $\widehat{\mathcal{CI}}(n)$ is a poset and admits an explicit  Möbius function, $\mu_{\widehat{\mathcal{CI}}(n)}([0_n,1_n])=(-1)^{n+1}(n-1)$, see \cite{diaz2022combinatorics}.
\end{remark}

\subsection{Cumulants}
Let $(\mathcal{A},\varphi)$ be a non-commutative probability space, and let
\(
(f_n)_{n\geq1}\), \(f_n:\mathcal{A}^n\to\mathbb{C},
\)
be a family of multilinear functionals. For each $n\geq1$ and each partition $\pi\in\mathcal{P}(n)$, we define an associated multilinear functional
\( f_\pi:\mathcal{A}^n\to\mathbb{C}
\)
as follows. If
\(
\pi=\{V_1,\ldots,V_r\}\in\mathcal{P}(n),
\)
where each block is written in increasing order as
\(
V_j=\{i_{j_1}<\cdots<i_{j_{|V_j|}}\},
\)
then we define
\begin{equation*}
f_\pi[a_1,\ldots,a_n]
:=f_{|V_1|}(a_{1_1},\ldots,a_{1_{|V_1|}})
    \cdots
    f_{|V_r|}(a_{r_1},\ldots,a_{r_{|V_r|}}).
\end{equation*}
The functionals $f_\pi$ extend the family $(f_n)_{n\geq1}$, since for the one-block partition $1_n\in\mathcal{P}(n)$ we recover
\begin{equation*}
f_{1_n}[a_1,\ldots,a_n]
=
f_n(a_1,\ldots,a_n),
\end{equation*}
for every $n\geq1$ and every $a_1,\ldots,a_n\in\mathcal{A}$. In particular, the linear functional $\varphi$ induces, for each $n\geq1$, the multilinear functional
\(
(a_1,\ldots,a_n)\mapsto \varphi(a_1\cdots a_n),
\)
so that $\varphi_\pi$ is well defined. Similarly, the extensions of the classical, Boolean, and free cumulants
\(
(c_n)_{n\geq1}\), \(
(b_n)_{n\geq1}\), \(
(r_n)_{n\geq1},
\)
are given by the families
\(
(c_\pi)_{\pi\in\mathcal{P}(n)}\), \((b_\pi)_{\pi\in\mathcal{I}(n)}\), \((r_\pi)_{\pi\in\mathcal{NC}(n)},
\)
which are defined analogously.
\begin{definition}
Let $(\mathcal A,\varphi)$ be a noncommutative probability space.
\begin{itemize}
    \item [(1) ]For each $n\geq 1$, the \textbf{multivariate classical cumulants}
    $(c_\pi)_{\pi\in\mathcal P(n)}$, studied in \cite{rota1964foundations,Schutzenberger1947}, are defined implicitly by
    \[
    \varphi_\pi[a_1,\dots,a_n]
    =
    \sum_{\substack{\sigma\in\mathcal P(n)\\ \sigma\leq \pi}}
    c_\sigma[a_1,\dots,a_n],
    \qquad \pi\in\mathcal P(n).
    \]

    \item [(2) ]For each $n\geq 1$, the \textbf{multivariate Boolean cumulants}
    $(b_\pi)_{\pi\in\mathcal I(n)}$, introduced in \cite{speicher1997boolean}, are defined implicitly by
    \[
    \varphi_\pi[a_1,\dots,a_n]
    =
    \sum_{\substack{\sigma\in\mathcal I(n)\\ \sigma\leq \pi}}
    b_\sigma[a_1,\dots,a_n],
    \qquad \pi\in\mathcal I(n).
    \]

    \item [(3) ]For each $n\geq 1$, the \textbf{multivariate free cumulants}
    $(r_\pi)_{\pi\in\mathcal{NC}(n)}$, introduced in \cite{speicher1994multiplicative}, are defined implicitly by
    \[
    \varphi_\pi[a_1,\dots,a_n]
    =
    \sum_{\substack{\sigma\in\mathcal{NC}(n)\\ \sigma\leq \pi}}
    r_\sigma[a_1,\dots,a_n],
    \qquad \pi\in\mathcal{NC}(n).
    \]
\end{itemize}
\end{definition}
The identities above are referred to as the moment--cumulant formulas. Let $\mu_{\mathcal P(n)}$, $\mu_{\mathcal I(n)}$, and $\mu_{\mathcal{NC}(n)}$ denote the Möbius functions on the lattices $\mathcal P(n)$, $\mathcal I(n)$, and $\mathcal{NC}(n)$, respectively. In particular, for the interval $[\sigma,1_n]$, one has
\begin{align*}
    &\mu_{\mathcal P(n)}(\sigma,1_n)=(-1)^{\#(\sigma)-1}(\#(\sigma)-1)!,\qquad \forall\ \sigma\in\mathcal{P}(n),\\ 
    &\mu_{\mathcal I(n)}(\sigma,1_n)=(-1)^{\#(\sigma)-1},\qquad \forall \ \sigma\in\mathcal{I}(n),\\
    &\mu_{\mathcal{NC}(n)}(\sigma,1_n)=(-1)^{\#(\sigma)-1}C_{\#(\sigma)-1},\qquad \forall \ \sigma\in\mathcal{NC}(n),
\end{align*}
where $C_m=\frac{1}{m+1}\binom{2m}{m}$ denotes the $m$-th Catalan number. More generally, the values $\mu(\sigma,\pi)$ for arbitrary intervals $[\sigma,\pi]$ factor as products of Möbius functions on smaller full lattices, since every such interval in the lattices considered here is isomorphic to a direct product of full lattices of lower orders, see \cite{aigner2012combinatorial,book:nica,stanley2011enumerative}. Therefore, by Möbius inversion, the relations in the previous definition can be inverted, yielding explicit expressions for the classical, Boolean, and free cumulants in terms of moments.
\begin{equation*}
    c_\pi[a_1,\ldots,a_n]=\sum_{\substack{ \sigma\in\mathcal{P}(n)\\ \sigma\leq \pi}} 
    \varphi_\sigma[a_1,\ldots,a_n]\,
    \mu_{\mathcal{P}(n)}(\sigma,\pi),\qquad \pi\in\mathcal{P}(n),
\end{equation*}
\begin{equation*}
        b_\pi[a_1,\ldots,a_n]=\sum_{\stackbin[\sigma\leq \pi]{\sigma\in\mathcal{I}(n)}{}}\varphi_\sigma[a_1,\ldots,a_n]\mu_{\mathcal{I}(n)}(\sigma,\pi),\qquad \pi\in\mathcal{I}(n),
    \end{equation*}
\begin{equation*}
        r_\pi[a_1,\ldots,a_n]=\sum_{\stackbin[\sigma\leq \pi]{\sigma\in\mathcal{NC}(n)}{}}\varphi_\sigma[a_1,\ldots,a_n]\mu_{\mathcal{NC}(n)}(\sigma,\pi),\qquad \pi\in\mathcal{NC}(n).
    \end{equation*}
The next theorem will play an important role in the following sections.
    \begin{theorem}[Lehner \cite{lehner2002free}]
        Let $\mu$ be a probability measure. Let $(m_n)_{n\in\mathbb{N}}$, $(c_n)_{n\in\mathbb{N}}$, $(b_n)_{n\in\mathbb{N}}$, and $(r_n)_{n\in\mathbb{N}}$ denote the moments, classical, Boolean and free cumulants of $\mu$, respectively. Then the following relations hold
\begin{equation}
    \label{41}
    r_n=\sum_{\pi\in\mathcal{P}_{conn}(n)} c_\pi,
\end{equation}
\begin{equation}
    \label{42}
    b_n=\sum_{\pi\in\mathcal{P}_{irr}(n)} c_\pi,
\end{equation}
\begin{equation}
    \label{43}
    b_n=\sum_{\pi\in\mathcal{NC}_{irr}(n)} r_\pi.
\end{equation}
    \end{theorem}
The previous theorem is stated in the univariate case. However, all three formulas remain valid in the multivariate setting. For the multivariate versions of \eqref{42} and \eqref{43}, see \cite{arizmendi2015relations}. The multivariate version of \eqref{41} was established by Belinschi and Nica in \cite{belinschi2008eta}.

\section{Markov-Krein correspondence}

We now recall the definition of the Markov--Krein correspondence, following Kerov’s approach via interlacing sequences and interlacing measures \cite{interlacingmeasures}.

\begin{definition}\label{rayleigh measure}
    Let $\tau$ be a bounded measure on $\mathbb R$, not necessarily positive. We say that $\tau$ is a \textbf{Rayleigh measure} if the following three conditions are satisfied: \(\displaystyle 0 \leq \tau\bigl((-\infty,u)\bigr) \leq 1\), for every \(u \in \mathbb{R}\); \(\displaystyle \tau\bigl((-\infty,\infty)\bigr) = 1;\) and \(\displaystyle \int_{-\infty}^{\infty} \ln(1 + |u|)\,\tau(du) < \infty.\)
We will denote $\mathcal F(\mathbb R)$ the space of Rayleigh measures on $\mathbb R$.
\end{definition}
\begin{remark}
    Using the previous definition, we can make the following observations: Two positive measures interlace\footnote{For the definition of interlacing measures, see \cite{interlacingmeasures}.} if and only if their difference is a Rayleigh measure; For every Rayleigh measure $\tau \in \mathcal F(\mathbb R)$, its Jordan decomposition $\tau = \tau_+ - \tau_-$ implies that $\tau_+$ and $\tau_-$ are two interlacing measures.
\end{remark}
\begin{definition}
    Let $F:\mathbb R\to \mathbb R $ be a measurable function. We say that $F$ is a \textbf{Rayleigh function} if $0\leq F(u)\leq 1$ for every $u\in\mathbb{R}$, $\displaystyle \int_{-\infty}^0\frac{F(u)du}{1+|u|}<\infty $ and $\displaystyle \int_{0}^\infty\frac{1-F(u)}{1+u}du<\infty$. 
\end{definition}
\begin{theorem}[Theorem 2.2.3 \cite{interlacingmeasures}]
    For every Rayleigh function $F$, there exists a unique probability measure $\mu$ such that
    \begin{equation}
    \label{I.maestra}
        \int_{-\infty}^{\infty}\frac{\mu(du)}{z-u}=\frac{1}{z}\exp \left(-\int_{-\infty}^0\frac{F(u)}{z-u}du+\int_{0}^\infty\frac{1-F(u)}{z-u}du\right),
    \end{equation}
    for all $z\in\mathbb{C}\backslash\mathbb{R}$.
\end{theorem}
If a Rayleigh function $F$ has bounded variation, then it can be regarded as the distribution function of some Rayleigh measure $\tau$ with bounded variation on $\mathbb{R}$. In this case, using integration by parts we get
\begin{equation*}
        \frac{1}{z}\exp\left(-\int_{-\infty}^0\frac{F(u)}{z-u}du+\int_{0}^\infty\frac{1-F(u)}{z-u}du\right)=\exp\left(-\int_{-\infty}^\infty\ln(z-u)\tau(du)\right).
    \end{equation*}
leading to the following definition:
\begin{definition}
Let $\mu$ be a probability measure on $\mathbb{R}$ and let $\tau$ be a Rayleigh measure with bounded variation on $\mathbb{R}$. The \textbf{Markov--Krein correspondence} is defined by the identity
\begin{equation}\label{2}
    \int_{-\infty}^{\infty}\frac{\mu(du)}{z-u}
    =
    \exp\left(
        -\int_{-\infty}^{\infty}\ln(z-u)\,\tau(du)
    \right),
    \qquad \Im(z)>0.
\end{equation}
\end{definition}
Let $\mathcal{M}(\mathbb{R})$ denote the space of probability measures on $\mathbb{R}$, and let $\mathcal{F}(\mathbb{R})$ denote the space of Rayleigh measures with bounded variation on $\mathbb{R}$. Then the map \(
\text{MK}:\mathcal{F}(\mathbb{R})\longrightarrow \mathcal{M}(\mathbb{R})\), \(\tau\longmapsto \mu
\), defined through \eqref{2}, establishes a homeomorphism between $\mathcal{F}(\mathbb{R})$ and $\mathcal{M}(\mathbb{R})$, when both spaces are endowed with their corresponding weak topologies.
Throughout the remainder of this paper, $\tau$ will denote a Rayleigh measure with bounded variation on $\mathcal{F}(\mathbb{R})$, while $\mu$ will denote a probability measure in $\mathcal{M}(\mathbb{R})$. We write $\text{MK}(\tau)$ for the \textbf{Markov--Krein transform} of $\tau$, and $\text{MK}^{(-1)}(\mu)$ for the \textbf{inverse Markov--Krein transform} of $\mu$. Following Kerov's convention in \cite{interlacingmeasures}, and in order to simplify the notation, we will omit the explicit use of $\text{MK}(\tau)$ and $\text{MK}^{(-1)}(\mu)$ whenever no confusion can arise. Thus, we will often write simply $\mu$ in place of $\text{MK}(\tau)$ and $\tau$ in place of $\text{MK}^{(-1)}(\mu)$.
 
Note that the Cauchy transform of $\tau$ admits the representation
\[
    \int_{-\infty}^{\infty}\frac{\tau(du)}{z-u}
    =
    \frac{d}{dz}\int_{-\infty}^{\infty}\ln(z-u)\,\tau(du).
\]
Therefore, identity \eqref{2} may be rewritten as
\[
    -\frac{d}{dz}\log\left(\int_{-\infty}^{\infty}\frac{\mu(du)}{z-u}\right)
    =
    \int_{-\infty}^{\infty}\frac{\tau(du)}{z-u}.
\]
Equivalently, in terms of the Cauchy transforms $G_\mu$ and $G_\tau$, the Markov--Krein correspondence is characterized by the differential equation
\begin{equation}\label{3}
    G_\mu'(z)+G_\mu(z)G_\tau(z)=0.
\end{equation}

\begin{example}\label{ejemplo-1}
    Consider $\mu$ to be the arcsine distribution supported on $(-2,2)$, that is, a probability measure with density
\(   \mu(dt)=\frac{1}{\pi\sqrt{4-t^2}}\mathds{1}_{\{-2\leq t\leq 2\}}\,dt.
\)
Its Cauchy transform and its derivative are given by
\[
    G_\mu(z)=\frac{1}{\sqrt{z^2-4}},
    \qquad
    G_\mu'(z)=-\frac{z}{(z^2-4)\sqrt{z^2-4}}.
\]
Therefore,
\[
    G_\tau(z)
    =
    -\frac{G_\mu'(z)}{G_\mu(z)}
    =
    \frac{z}{z^2-4},
\]
which coincides with the Cauchy transform of a symmetric Bernoulli distribution supported on $\{-2,2\}$.

\end{example}

\begin{example}\label{ejemplo-2}
    Consider $\mu$ to be the semicircular distribution, that is, a probability measure with density
\(
    \mu(dt)=\frac{\sqrt{4-t^2}}{2\pi}\mathds{1}_{\{-2\leq t\leq 2\}}\,dt.
\)
Its Cauchy transform and its derivative are given by
\[
    G_\mu(z)=\frac{z-\sqrt{z^2-4}}{2},
    \qquad
    G_\mu'(z)=\frac{1}{2}\left(1-\frac{z}{\sqrt{z^2-4}}\right).
\]
Therefore,
\[
    G_\tau(z)
    =
    -\frac{G_\mu'(z)}{G_\mu(z)}
    =
    \frac{1}{\sqrt{z^2-4}}\mathds{1}_{\{-2\leq t\leq 2\}}.
\]
This coincides with the Cauchy transform of the arcsine distribution supported on $(-2,2)$.
\end{example}

\begin{example}\label{example-3}
    Let $\mu=\frac12\delta_{-2}+\frac12\delta_2$ be the symmetric Bernoulli distribution supported on $\{-2,2\}$. Its Cauchy transform and derivative are
    \begin{equation*}
        G_\mu(z)=\frac{z}{z^2-4},
        \qquad
        G_\mu'(z)=-\frac{z^2+4}{(z^2-4)^2}.
    \end{equation*}
    Therefore,
    \begin{equation*}
        G_\tau(z)
        =
        -\frac{G_\mu'(z)}{G_\mu(z)}
        =
        \frac{z^2+4}{z(z^2-4)}
        =
        \frac{2z}{z^2-4}-\frac{1}{z}
        =
        2G_\mu(z)-G_{\delta_0}(z).
    \end{equation*}
    Hence, \(\tau=2\mu-\delta_0\).
\end{example}

We will regard identity \eqref{3} as the main tool for identifying the Markov--Krein transform \(\mu\) of a signed measure \(\tau\). The reason for this is that it is expressed in terms of the Cauchy transform, which is one of the most fundamental analytic tools in non-commutative probability. With this idea in mind, if we define

\begin{align*}
    &M_\mu(z):=\frac{1}{z}G_\mu\left(\frac{1}{z}\right),\qquad F_\mu(z):=\frac{1}{G_\mu(z)},\qquad B_\mu(z):=\frac{1}{z}-F_\mu\left(\frac{1}{z}\right).
\end{align*}
After straightforward computations, if $\tau\in \mathcal{F}(\mathbb{R})$ is a Rayleigh measure and $\mu\in \mathcal{M}(\mathbb{R})$ be its Markov--Krein transform. Then the following are equivalent
\begin{equation*}
    \frac{G_\mu^\prime(z)}{G_\mu(z)}=-G_\tau(z),\qquad  \frac{M_\mu^\prime(z)}{M_\mu(z)}=\frac{1}{z}[M_\tau(z)-1],\qquad \frac{F_\mu^\prime(z)}{F_\mu(z)}=\frac{1}{F_\tau(z)},\qquad\frac{1+z^2B_\mu^\prime(z)}{\frac{1}{z}-B_\mu(z)}=\frac{1}{\frac{1}{z}-B_\tau(z)}.
\end{equation*}
In the case where the Markov-Krein transform can be applied iteratively, after carrying out the computations and following the relations
\begin{equation*}
    G_{\text{MK}^{(-1)}(\mu)}(z)=-\frac{G_\mu^\prime(z)}{G_\mu(z)},\qquad G_{\text{MK}^{(-1)}(\text{MK}^{(-1)}(\mu))}(z)=-\frac{G_{\text{MK}^{(-1)}(\mu)}^\prime(z)}{G_{\text{MK}^{(-1)}(\mu)}(z)},
\end{equation*}
then
\begin{equation}\label{cauchy-MK(MK(mu))}
    G_{\text{MK}^{(-1)}(\text{MK}^{(-1)}(\mu))}(z)=\frac{G_\mu^\prime(z)}{G_\mu(z)}-\frac{G_\mu^{\prime\prime}(z)}{G_\mu^\prime(z)}.
\end{equation}
\begin{remark}
It is important to note that the inverse of the Markov--Krein transform is not, in general, closed under iteration. More precisely, if $\mu \in \mathcal{M}(\mathbb{R})$ is a probability measure, then its inverse image $\text{MK}^{(-1)}(\mu)$ is a Rayleigh measure in $\mathcal{F}(\mathbb{R})$. However, in order for the second iterate \(\text{MK}^{(-1)}(\text{MK}^{(-1)}(\mu))\) to be well-defined, one must ensure that $\text{MK}^{(-1)}(\mu)$ is itself a probability measure. This is not always the case, since a Rayleigh measure need not be a probability measure, as seen in the Example \ref{example-3}. Characterizing when this second iteration is well-defined remains, to the best of our knowledge, an unknown.
\end{remark}
\subsection{Moments of the Makov--Krein correspondence}
Let \(\mu\in \mathcal{M}(\mathbb{R})\) be a probability measure on \(\mathbb{R}\), and let \(\tau\in \mathcal{F}(\mathbb{R})\) be a Rayleigh measure on \(\mathbb{R}\) satisfying identity \eqref{2}. We define the \(n\)-th moments of \(\mu\) and \(\tau\) by
\begin{equation*}
    h_n:=\int_{-\infty}^\infty t^n\,\mu(dt),\qquad p_n:=\int_{-\infty}^\infty t^n\,\tau(dt).
\end{equation*}

\begin{theorem}[Theorem 3.2.3 \cite{interlacingmeasures}]
Suppose that a Rayleigh measure \(\tau\in \mathcal{F}(\mathbb{R})\) and a probability measure \(\mu\in\mathcal{M}(\mathbb{R})\) satisfy identity \eqref{2}. Then
\begin{equation}
    \label{8}
    \sum_{n=0}^\infty h_n z^n=\exp\left(\sum_{n=1}^\infty \frac{p_n}{n}z^n\right).
\end{equation}
\end{theorem}
For a sequence \(\{m_n\}_{n=0}^{\infty}\), the moment problem asks whether there exists a measure \(\rho\) such that
\begin{equation*}
    m_n = \int_\mathbb{R} x^n \, \rho(dx), \qquad \text{for each } n = 0, 1, 2, \ldots
\end{equation*}
Solving the moment problem consists in determining the existence, uniqueness, and construction of such a measure \(\rho\). This problem has significant applications in probability theory, statistics, functional analysis, and other areas of mathematics \cite{shohat1943problem}. In this sense, the moment problem for the Rayleigh measure \(\tau\in\mathcal{F}(\mathbb{R})\) is determinate if and only if the moment problem for its Markov--Krein transform \(\mu\in\mathcal{M}(\mathbb R)\) is determinate. Moreover, if the measures \(\tau\) and \(\mu\) are uniquely determined by their moments, then identity \eqref{3} is equivalent to identity \eqref{8}.

\begin{theorem}[\cite{interlacingmeasures}]\label{thm-momentos de h y p}
If \(p_n\) and \(h_n\) are the moments of a Rayleigh measure \(\tau\) and of its Markov--Krein transform \(\mu\), respectively. Then the following are equivalent:
\begin{itemize}
    \item[(1)] \(\displaystyle nh_n=\sum_{k=1}^n h_{n-k}p_k\);
    \item[(2)] \(\displaystyle h_n=\sum_{\rho\vdash n}\frac{p_1^{r_1}p_2^{r_2}\cdots p_n^{r_n}}{1^{r_1}r_1!2^{r_2}r_2!\cdots n^{r_n}r_n!}\), where the sum runs over all integer partitions \(\rho=(1^{r_1},2^{r_2},\ldots,n^{r_n})\) such that \(n=r_1+2r_2+\cdots+nr_n\) (Girard--Newton identity);
    \item[(3)] \(\displaystyle h_n=\frac{1}{n!}\sum_{\pi\in S_n}\prod_{k\geq 1}p_k^{r_k(\pi)}\), where \(S_n\) denotes the symmetric group.
\end{itemize}
\end{theorem}
A natural question that arises in this context is whether the formulas presented in the previous relations can be inverted, that is, whether one can find a relation expressing \(p_n\) in terms of the moments \(h_n\). A formula presented in \cite{interlacingmeasures} is the following:
\begin{equation*}
    p_n=(-1)^{n-1}\begin{vmatrix}
        h_1& 1& 0&\cdots &0\\
        2h_2&h_1&1&\cdots &0\\ 
        \vdots &\vdots&\vdots &&\vdots \\
        nh_{n}&h_{n-1}&h_{n-2}&\cdots &h_1
    \end{vmatrix}.
\end{equation*}
From relation \eqref{8}, one can also derive a formula expressing \(p_n\) in terms of \(h_n\) using partitions, as stated in the following proposition. In Section~\ref{moment-moment-section}, Theorem~\ref{moment-moment}, we present an alternative formula in terms of cyclic interval partitions, which is of course equivalent. We emphasize, however, that the latter formulation admits a multivariate generalization.
\begin{proposition}\label{p_n-h_n-particiones}
If \(p_n\) and \(h_n\) are the moments of a Rayleigh measure \(\tau\) and of its Markov--Krein transform \(\mu\), respectively. Then 
    \begin{equation*}
        p_n=\sum_{\pi\in\mathcal{P}(n)}h_\pi(-1)^{\#(\pi)-1}\frac{\#(\pi)!(\#(\pi)-1)!}{(n-1)!}.
    \end{equation*}
\end{proposition}
\begin{proof}
Since $\mu$ is the Markov--Krein transform of $\tau$, we have
\[
    \sum_{n=0}^\infty h_n z^n
    =
    \exp\left(\sum_{n=1}^\infty \frac{p_n}{n}z^n\right),
\]
which is equivalent to
\[
    \log\left(\sum_{n=0}^\infty \frac{h_n n!}{n!} z^n\right)
    =
    \sum_{n=1}^\infty \frac{p_n (n-1)!}{n!}z^n.
\]
If we define $\widetilde{h}_n := h_n n!$ and $\widetilde{p}_n := p_n (n-1)!$, then
\[
    \log\left(\sum_{n=0}^\infty \frac{\widetilde{h}_n}{n!} z^n\right)
    =
    \sum_{n=1}^\infty \frac{\widetilde{p}_n}{n!}z^n.
\]
This implies the identity
\[
    \widetilde{h}_n
    =
    \sum_{\pi\in\mathcal{P}(n)} \widetilde{p}_\pi,
\]
and, by Möbius inversion, we obtain
\[
    \widetilde{p}_n
    =
    \sum_{\pi\in\mathcal{P}(n)}
    \widetilde{h}_\pi
    (-1)^{\#(\pi)-1} (\#(\pi)-1)!.
\]
Substituting the definitions of $\widetilde{h}_n$ and $\widetilde{p}_n$, we obtain
\begin{equation}
    p_n
    =
    \sum_{\pi\in\mathcal{P}(n)}
    h_\pi
    (-1)^{\#(\pi)-1}
    \frac{\#(\pi)!(\#(\pi)-1)!}{(n-1)!}.
\end{equation}
\end{proof}
The previous formulas may be used to prove well-known properties of the Markov--Krein correspondence. For instance, if $\mu$ is the Markov--Krein transform of a Rayleigh measure $\tau$, then any translated and rescaled version of $\mu$ corresponds to the same transformation applied to $\tau$. Also,  $\mu$ is supported on a compact interval $[-a,b]$ if and only if the associated Rayleigh measure $\tau$ is also supported on $[-a,b]$. Moreover, when all measures involved have compact support, both the Markov–Krein transform
and its inverse are continuous with respect to weak convergence.

\section{Main Results}

\subsection{Multivariate Markov--Krein correspondence}

Let $\mathcal{H}$ be a separable Hilbert space, and let $\mathcal{B}(\mathcal{H})$ denote the $\ast$-algebra of all bounded linear operators on $\mathcal{H}$ and $\mathcal{A}$ be a $\ast$-subalgebra of $\mathcal{B}(\mathcal{H})$. Given a unit vector $\xi \in \mathcal{H}$, let 
$p:\mathcal{H} \to \mathcal{H}$ be the orthogonal projection onto the one-dimensional subspace generated by $\xi$, and define $q = \mathrm{id}_{\mathcal{H}} - p$ as the orthogonal projection onto the hyperplane $\mathcal{H}_\xi := \{\eta \in \mathcal{H} : \langle \eta, \xi \rangle = 0\}$. Explicitly, for every $\eta \in \mathcal{H}$,
\[
p(\eta) = \langle \eta, \xi \rangle \xi,
\qquad
q(\eta) = \eta - \langle \eta, \xi \rangle \xi.
\]
\begin{definition}\label{mom-mult-h} 
For \(n\ge1\), the functional
\(
h_n:\mathcal A^n\to\mathbb C
\)
defined by
\[
h_n(A_1,\ldots,A_n)
=
\langle A_1\cdots A_n\xi,\xi\rangle,
\]
is called the \(n\)-th \textbf{multivariate moment of the Markov--Krein transform} of $h$.
\end{definition}
\begin{definition}\label{mom-mult-p}
For \(n\ge1\), the functional
\(
p_n:\mathcal A^n\to\mathbb C
\)
defined by
\[
p_n(A_1,\ldots,A_n)
=
\operatorname{Tr}
\big(
A_1\cdots A_n
-
qA_1q\cdots qA_nq
\big)
\]
is called the \(n\)-th \textbf{multivariate moment of the inverse
Markov--Krein transform} of $h$.
\end{definition}
The functional \(h_n\) is well defined, as it corresponds to the vector state associated with \(\xi\). Moreover, \(p_n\) is well defined since the operator \(a_1 \cdots a_n - q a_1 q \cdots q a_n q\) is trace class. These definitions are consistent with, and extend, the univariate case. Indeed, in the particular case where \(a := a_1 = a_2 = \cdots = a_n\), it follows that
\begin{equation*}
    p_n(a,\ldots,a)=\operatorname{Tr}(a^n-qaq\cdots qaq)=\operatorname{Tr}(a^n-(qaq)^n)=p_n,
\end{equation*}
and
\begin{equation*}
    h_n(a,\ldots,a)=\langle a\cdots a\xi,\xi\rangle =\langle a^n\xi,\xi\rangle=h_n.    
\end{equation*} 
where in the univariate case $p_n=\operatorname{Tr}(a^n-(qaq)^n)$ and $h_n=\langle a^n\xi,\xi\rangle$ correspond respectively to the \(n\)-th moment of the Markov--Krein transform and of its inverse, as proved by M. G. Krein \cite{Krein-teoria-operadores}.
\begin{remark}Some important properties; the functional $p_n$ is cyclically invariant as a consequence of traciality of the usual trace, it means, $p_n(A_1,\ldots,A_{n-1},A_n)=p_n(A_2,\ldots,A_n,A_1)$. In contrast, $h_n$ is not necessarily cyclically invariant. For example, if
\[
A=
\begin{bmatrix}
1&1\\
0&0
\end{bmatrix},
\qquad
B=
\begin{bmatrix}
1&0\\
1&0
\end{bmatrix},
\]
then
\(
h_2(A,B)\neq h_2(B,A).
\)
Moreover,
\(
h_n(A_1,\ldots,A_n)
=
h_1(A_1\cdots A_n),
\)
whereas, in general, we have 
\(
p_n(A_1,\ldots,A_n)
\neq
p_1(A_1\cdots A_n),
\)
as can be verified using the previous matrices $A$, $B$ and $\xi=(1,0)$, since \(p_2(B,A)\neq p_1(BA)
\). Finally, if we consider the vector $\xi=(1,0,\ldots,0)$ and matrices $A_1,\ldots ,A_n\in M_{n\times n}(\mathbb{C})$ defined as follows 
\begin{equation*}
    A_k=\begin{bmatrix}
        a_{11}^{(k)}& r_k\\
        c_k& \widetilde{A}_k
    \end{bmatrix},\qquad r_k\in M_{1\times (n-1)}(\mathbb{C}),\quad c_k\in M_{(n-1)\times 1}(\mathbb{C}),\quad \widetilde{A}_k\in M_{(n-1)\times (n-1)}(\mathbb{C}),
\end{equation*}
and consider the products
\begin{equation*}
    A_1\cdots A_n=\begin{bmatrix}
        x_{1,1}&\cdots&x_{1,n}\\
        \vdots & \ddots & \vdots \\
        x_{n,1}&\cdots & x_{n,n}
    \end{bmatrix},\qquad \widetilde{A}_1\cdots \widetilde{A}_n=\begin{bmatrix}
        b_{1,1}&\cdots&b_{1,(n-1)}\\
        \vdots & \ddots & \vdots \\
        b_{(n-1),1}&\cdots & b_{(n-1),(n-1)}
    \end{bmatrix},
\end{equation*}
for some $x_{i,j},b_{k,l}\in\mathbb{C}$. Since $\xi=(1,0,\ldots,0)$
\begin{equation*}
    p=\begin{bmatrix}
        1&0&\cdots &0\\ 
        0 &0 &\cdots &0\\ 
        \vdots &\vdots &&\vdots \\
        0&0&\cdots &0
    \end{bmatrix}\qquad \text{and}\qquad q=\begin{bmatrix}
        0&0&\cdots &0\\ 
        0 &1 &\cdots &0\\ 
        \vdots &\vdots &&\vdots \\
        0&0&\cdots &1
    \end{bmatrix},
\end{equation*}
then
\begin{align*}
    qA_1q\cdots qA_nq=&\begin{bmatrix}
        0&0\\ 
        0&I_{n-1}
    \end{bmatrix}\begin{bmatrix}
        a_{11}^{(k)}& r_k\\
        c_k& \widetilde{A}_k
    \end{bmatrix}\begin{bmatrix}
        0&0\\ 
        0&I_{n-1}
    \end{bmatrix}\cdots \begin{bmatrix}
        0&0\\ 
        0&I_{n-1}
    \end{bmatrix}\begin{bmatrix}
        a_{11}^{(n)}& r_n\\
        c_n& \widetilde{A}_n
    \end{bmatrix}\begin{bmatrix}
        0&0\\ 
        0&I_{n-1}
    \end{bmatrix}\\
    =&\begin{bmatrix}
        0 & \vec{0} \\ 
        \vec{0}^T & \widetilde{A}_1\cdots \widetilde{A}_n
    \end{bmatrix}.
\end{align*}
Hence
\begin{equation*}
   h_n(A_1,\ldots,A_n)=x_{1,1}\qquad \text{and}\qquad  p_n(A_1,\ldots,A_n)=\sum_{i=1}^nx_{i,i}-\sum_{i=1}^{n-1}b_{i,i}.
\end{equation*}
\end{remark}
\begin{proposition}
The functionals $h_n$ and $p_n$ are multilinear.
\end{proposition}

\begin{proof}
Let $\lambda \in \mathbb{C}$ and let $A_1,\ldots,A_i,\widetilde{A}_i,\ldots,A_n \in \mathcal{A}$. First, using the linearity of the inner product,
\begin{align*}
\langle A_1 &\cdots (A_i + \lambda \widetilde{A}_i) \cdots A_n \xi, \xi \rangle
\\
&=
\langle A_1 \cdots A_i \cdots A_n \xi, \xi \rangle
+
\lambda
\langle A_1 \cdots \widetilde{A}_i \cdots A_n \xi, \xi \rangle,
\end{align*}
which proves the multilinearity of $h_n$. Similarly, using the linearity of the trace,
\begin{align*}
&
\mathrm{Tr}
\left(
A_1 \cdots (A_i + \lambda \widetilde{A}_i) \cdots A_n
-
q A_1 q \cdots q (A_i + \lambda \widetilde{A}_i) q \cdots q A_n q
\right)
\\ 
&=\mathrm{Tr}\big(A_1\cdots A_i\cdots A_n+A_1\cdots \lambda \widetilde{A}_i\cdots A_n-qA_1q\cdots qA_iq\cdots qA_nq-qA_1q\cdots q\lambda \widetilde{A}_iq\cdots qA_nq\big)\\
&=\mathrm{Tr}\big(A_1\cdots A_i\cdots A_n\big)+\lambda\mathrm{Tr}\big(A_1\cdots \widetilde{A}_i\cdots A_n\big)\\ 
&\quad -\mathrm{Tr}\big(qA_1q\cdots qA_iq\cdots qA_nq\big)-\lambda\mathrm{Tr}\big(qA_1q\cdots q\widetilde{A}_iq\cdots qA_nq\big)\\
&=
p_n(A_1,\ldots,A_i,\ldots,A_n)
+
\lambda
p_n(A_1,\ldots,\widetilde{A}_i,\ldots,A_n),
\end{align*}
which proves the multilinearity of $p_n$.
\end{proof}
\subsection{Moment-moment formula}\label{moment-moment-section}

Definition \ref{mom-mult-h} can be extended to a family of multilinear functionals indexed by set partitions. For $\pi \in \mathcal{P}(n)$, define $h_\pi : \mathcal{A}^n \to \mathbb{C}$ by the following rule. If $\pi = \{V_1,\ldots,V_r\}$, then
\[
h_\pi[A_1,\ldots,A_n]
=
\prod_{k=1}^r
h_{|V_k|}
(A_{i_1^{(k)}},\ldots,A_{i_{|V_k|}^{(k)}}),
\]
where $V_k = \{i_1^{(k)},\ldots,i_{|V_k|}^{(k)}\}$. Analogously, Definition \ref{mom-mult-p} extends to multilinear functionals $p_\pi : \mathcal{A}^n \to \mathbb{C}$ defined by
\[
p_\pi[A_1,\ldots,A_n]
=
\prod_{k=1}^r
p_{|V_k|}
(A_{i_1^{(k)}},\ldots,A_{i_{|V_k|}^{(k)}}).
\]
The following three lemmas appear as results in \cite{arizmendi}. We provide detailed proofs using just the definitions of $h_n$ and $p_n$, as they are essential for the subsequent results. We denote by \((b^h_n)_{n\geq1}\) the family of Boolean cumulant with respect to the functional \(h:=h_1\).
\begin{lemma}
    Let \(A_1,\ldots,A_n\in\mathcal{A}\). We can express the \(n\)-th Boolean cumulant with respect to \(h\), \(b_n^h\), in the following way
    \begin{equation*}
         b_n^h(A_1,\ldots,A_n)=\operatorname{Tr}(pA_1q\cdots qA_np).
    \end{equation*}
\end{lemma}

\begin{proof}
     Since \(p\) is a projection onto \(\xi\in\mathcal{H}\), a unit vector, then
    \begin{equation*}
        h_n(A_1,\ldots,A_n)=\langle A_1\cdots A_n\xi,\xi\rangle=\operatorname{Tr}(pA_1\cdots A_np).
    \end{equation*}
    Then, since \(id_{\mathcal{H}}=p+q\), we can do the following
    \begin{align*}
        h_n(A_1,\ldots,A_n)=&\operatorname{Tr}(pA_1\cdots A_np)
        =\operatorname{Tr}(pA_1(p+q)A_2(p+q)\cdots (p+q)A_np),
    \end{align*}
    thus, by denoting \(p_{\delta_i}=p\) if \(\delta_i=0\) or \(p_{\delta_i}=q\) if \(\delta_i=1\), for each \(i=1,2,\ldots,n-1\), one has the following
    \begin{align*}
         h_n(A_1,\ldots,A_n)=\sum_{\delta_i\in\{0,1\}}\operatorname{Tr}(pA_1p_{\delta_1}A_2p_{\delta_2}\cdots p_{\delta_{n-1}}A_np).
    \end{align*}
    We define \(\widetilde{b}_k^h(A_1,\ldots,A_k):=\operatorname{Tr}(pA_1qA_2q\cdots qA_kp)\), for each \(k=1,2,\ldots,n\), and we will show that \(\widetilde{b}_k^h\) satisfies being the \(k\)-th Boolean cumulant with respect to \(h\). Observe that in the case \(\delta_j=1\) for each \(j=1,2,\ldots,n-1\), then we would have the term
    \begin{align*}
        \operatorname{Tr}(pA_1p_{\delta_1}A_2p_{\delta_2}\cdots p_{\delta_{n-1}}A_np)=&\operatorname{Tr}(pA_1p_1A_2p_1\cdots p_1A_np)\\ 
        =&\operatorname{Tr}(pA_1qA_2q\cdots qA_np)\\ 
        =&\widetilde{b}_n^h(A_1,\ldots,A_n).
    \end{align*}    
    In the case of having exactly one \(\delta_j=0\) for some \(1\leq j\leq n-1\), one has the following
    \begin{align*}
        \operatorname{Tr}(pA_1p_{\delta_1}\cdots p_{\delta_{j-1}}A_jp_{\delta_j}A_{j+1}p_{\delta_{j+1}}\cdots p_{\delta_{n-1}}A_np)=&\operatorname{Tr}(pA_1q\cdots qA_{j}pA_{j+1}q\cdots qA_np)\\
        =&\operatorname{Tr}(pA_1q\cdots qA_{j}p)\operatorname{Tr}(pA_{j+1}q\cdots qA_np)\\ 
        =&\widetilde{b}^h_{j}(A_1,\ldots,A_{j})\widetilde{b}^h_{n-j}(A_{j+1},\ldots,A_n).
    \end{align*}
    In this case, when there is exactly one \(\delta_j=0\), this can be interpreted as having an interval partition \(\pi\) of \([n]\) of the form \(\pi=\{\{A_1,\ldots,A_{j}\},\{A_{j+1},\ldots,A_n\}\}\), which is unique, that is, there is no other interval partition of \([n]\) by means of which we can interpret the previous expression. The same analogy can be made for any number of \(\delta_j=0\), and this will give us a unique interval partition of \([n]\) by means of which we can interpret the element
    \begin{equation*}
        \operatorname{Tr}(pA_1p_{\delta_1}\cdots p_{\delta_{j-1}}A_jp_{\delta_j}A_{j+1}p_{\delta_{j+1}}\cdots p_{\delta_{n-1}}A_np).
    \end{equation*}
    From the above, it follows that there is a bijection between the choices of \(\delta_i\) and the interval partitions of \([n]\), from which we can conclude that 
    \begin{align}
        h_n(A_1,\ldots,A_n)=&\sum_{\delta_i\in\{0,1\}}\operatorname{Tr}(pA_1p_{\delta_1}A_2p_{\delta_2}\cdots p_{\delta_{n-1}}A_np)\label{33}\\ 
        =&\sum_{\pi\in I(n)}\widetilde{b}^h_\pi[A_1,\ldots,A_n]\label{34},
    \end{align}
    where if \(\pi=\{V_1,\ldots,V_r\}\), one has \(\widetilde{b}^h_\pi[A_1,\ldots,A_n]=\widetilde{b}^h_{|V_1|}(A_1,\ldots,A_{|V_1|})\cdots \widetilde{b}^h_{n-|V_r|}(A_{|V_r|+1},\ldots,A_n)\). In order to prove that the Boolean cumulants with respect to \(h\) satisfy \(b^h_n(A_1,\ldots,A_n)=\widetilde{b}^h_n(A_1,\ldots,A_n)\), we will do it by induction over $n$. The case \(n=1\) holds since
    \begin{equation*}
        b_1^h(A_1)=h_1(A_1)=\operatorname{Tr}(pA_1p)=\widetilde{b}_1^h(A_1).
    \end{equation*}
    Suppose that it is satisfied that \(b^h_k(A_1,\ldots,A_k)=\widetilde{b}^h_k(A_1,\ldots,A_k)\), for each \(k=1,2,\ldots,n\). By the formula of Boolean moments-cumulants with respect to \(h\), and by applying the induction hypothesis
    \begin{align*}
        h_{n+1}(A_1,\ldots,A_{n+1})=&\sum_{\pi\in I(n+1)}b^h_\pi[A_1,\ldots,A_{n+1}]\\ 
        =&b^h_{n+1}(A_1,\ldots,A_{n+1})+\sum_{\stackbin[\pi\neq1_{n+1}]{}{\pi\in I(n+1)}}b^h_\pi[A_1,\ldots,A_{n+1}]\\ 
        =&b^h_{n+1}(A_1,\ldots,A_{n+1})+\sum_{\stackbin[\pi\neq1_{n+1}]{}{\pi\in I(n+1)}}\widetilde{b}^h_\pi[A_1,\ldots,A_{n+1}].
    \end{align*}
    On the other hand, by means of expression \eqref{34} we can say that
    \begin{equation*}
         h_{n+1}(A_1,\ldots,A_{n+1})=\widetilde{b}^h_{n+1}(A_1,\ldots,A_{n+1})+\sum_{\stackbin[\pi\neq1_{n+1}]{}{\pi\in I(n+1)}}\widetilde{b}^h_\pi[A_1,\ldots,A_{n+1}],
    \end{equation*}
    hence \(\widetilde{b}^h_{n+1}(A_1,\ldots,A_{n+1})=b^h_{n+1}(A_1,\ldots,A_{n+1})\). Thus concluding the induction argument. 
\end{proof}

\begin{lemma}\label{lema-mom-cum-booleanos}
    Let \(A_1,\ldots,A_n\in\mathcal{A}\). We can express the \(n\)-th multivariate moment \(p_n\) in terms of the Boolean cumulants \(b^h\) in the following way
    \begin{equation}\label{momentos_tau_cumulantes_booleanos_mu}
        p_n(A_1,\ldots,A_n)=\sum_{l=1}^n\quad \sum_{1\leq i(1)<\cdots<i(l)\leq n}\quad \prod_{k=1}^lb^h_{i(k+1)-i(k)}(A_{i(k)},A_{i(k)+1},\ldots,A_{i(k+1)-1}),
    \end{equation}
    with the convention that \(i(l+1)=i(1)\).
\end{lemma}
This lemma introduces a new notation in which the coefficients are taken cyclically. This notation appears in the terms \(b^h_{i(k+1)-i(k)}(A_{i(k)},A_{i(k)+1},\ldots,A_{i(k+1)-1})\) when \(k=l\), that is, the term \(b^h_{i(1)-i(l)}(A_{i(l)},A_{i(l)+1},\ldots,A_{i(1)-1})\) denotes the term
\begin{equation*}
    b^h_{n+i(1)-i(l)}(A_{i(l)},A_{i(l)+1},\ldots,A_{n-1},A_n,A_1,A_2,\ldots,A_{i(1)-1}).
\end{equation*}

\begin{proof}
    Let us use the same trick of multiplying by \(id_{\mathcal{H}}=p+q\) that we used in the previous lemma, but now in the following way
    \begin{align*}
        p_n(A_1,\ldots,A_n)=&\operatorname{Tr}(A_1\cdots A_n-qA_1q\cdots qA_nq)\\
        =&\operatorname{Tr}((p+q)A_1(p+q)\cdots (p+q)A_n(p+q)-qA_1q\cdots qA_nq)\\
        =&\operatorname{Tr}\left(\sum_{\delta_i \in\{0,1\}}p_{\delta_1}A_1p_{\delta_2}\cdots p_{\delta_n}A_np_{\delta_{n+1}}-qA_1q\cdots qA_nq\right),
    \end{align*}
    where \(p_{\delta_i}=p\) if \(\delta_i=0\) and \(p_{\delta_i}=q\) if \(\delta_i=1\), for each \(i=1,2,\ldots,n+1\). Noting that when \(\delta_i=1\), for every \(i=1,2,\ldots,n+1\), the term \(qA_1q\cdots qA_n q\) is cancelled, and by linearity of the trace one has
    \begin{equation*}
         p_n(A_1,\ldots,A_n)=\sum_{\stackbin[(\delta_1,\ldots,\delta_n)\neq(1,\ldots,1)]{}{\delta_i \in\{0,1\}}}\operatorname{Tr}\left(p_{\delta_1}A_1p_{\delta_2}\cdots p_{\delta_n}A_np_{\delta_{n+1}}\right) .
    \end{equation*}
    Moreover, when \(p_{\delta_1}\neq p_{\delta_{n+1}}\), one has \(\operatorname{Tr}(p_{\delta_1}A_1p_{\delta_2}\cdots p_{\delta_n}A_np_{\delta_{n+1}})=0\). Therefore, the previous expression can be written as
    \begin{align}\label{moment-transpuesta}
         p_n(A_1,\ldots,A_n)=&\sum_{\stackbin[(\delta_1,\ldots,\delta_n)\neq(1,\ldots,1)]{}{\delta_i \in\{0,1\}}}\operatorname{Tr}(p_{\delta_1}A_1p_{\delta_2}\cdots p_{\delta_n}A_np_{\delta_{1}}),
    \end{align}
    the condition \((\delta_1,\ldots,\delta_n)\neq(1,\ldots,1)\) tells us that there is at least one \(p_{\delta_k}=p\). In the case in which for exactly one \(k\in\{1,2,\ldots,n\}\) it is satisfied that \(p_{\delta_k}=p\), then one has 
    \begin{align*}
        \operatorname{Tr}(p_{\delta_1}A_1p_{\delta_2}\cdots A_{k-1}p_{\delta_k}A_k\cdots  p_{\delta_n}A_np_{\delta_{1}})=&\operatorname{Tr}(qA_1q\cdots A_{k-1}pA_k\cdots  qA_nq)\\ 
        =&\operatorname{Tr}(qA_1q\cdots A_{k-1}p^2A_k\cdots  qA_nq)\\ 
        =&\operatorname{Tr}(pA_k\cdots qA_nq A_1q\cdots A_{k-1}p)\\ 
        =&b_{n}^\mu (A_k,\ldots,A_{k-1}).
    \end{align*}
    In the case of having exactly two indices \(i(1)\) and \(i(2)\), with \(i(1)<i(2)\), that satisfy \(p_{\delta_{i(1)}}=p\) and \(p_{\delta_{i(2)}}=p\), one would have the following
    \begin{align*}
        \operatorname{Tr}(p_{\delta_1}A_1p_{\delta_2}&\cdots A_{k-1}p_{\delta_k}A_k\cdots  p_{\delta_n}A_np_{\delta_{1}})\\
        =&\operatorname{Tr}(qA_1q\cdots A_{i(1)-1}pA_{i(1)} \cdots A_{i(2)-1}pA_{i(2)}\cdots qA_nq)\\ 
        =&\operatorname{Tr}(qA_1q\cdots A_{i(1)-1}p^2A_{i(1)} \cdots A_{i(2)-1}p^2A_{i(2)}\cdots qA_nq)\\
        =&\operatorname{Tr}(pA_{i(1)} \cdots A_{i(2)-1}p^2A_{i(2)}\cdots qA_nqA_1q\cdots A_{i(1)-1}p)\\ 
        =&\operatorname{Tr}(pA_{i(1)} \cdots A_{i(2)-1}p)\operatorname{Tr}(pA_{i(2)}\cdots qA_nqA_1q\cdots A_{i(1)-1}p)\\ 
        =&b_{i(2)-i(1)}^h(A_{i(1)},\ldots,A_{i(2)-1})b_{n+i(1)-i(2)}^h(A_{i(2)},\ldots,A_{i(1)-1}).
    \end{align*}
    From these two particular cases we note that each one can be expressed as a product of Boolean cumulants of \(h\). In a general case, considering all possible choices of \(\delta_i\), one has that expression \eqref{moment-transpuesta} can be viewed in the following way
    \begin{align*}
        p_n(A_1,\ldots,A_n)=&\sum_{\stackbin[(\delta_1,\ldots,\delta_n)\neq(1,\ldots,1)]{}{\delta_i \in\{0,1\}}}\operatorname{Tr}(p_{\delta_1}A_1p_{\delta_2}\cdots p_{\delta_n}A_np_{\delta_{1}})\\ 
        =&\sum_{l=1}^n\quad \sum_{1\leq i(1)<\cdots<i(l)\leq n}\quad \prod_{k=1}^lb^h_{i(k+1)-i(k)}(A_{i(k)},A_{i(k)+1},\ldots,A_{i(k+1)-1}),
    \end{align*}
    where the first sum represents the number of \(\delta_{i}\) such that \(p_{\delta_i}=p\), and the second sum runs over all possible choices of those values between \(1\) and \(n\).
\end{proof} 

\begin{lemma}
    Let \(A_1,\ldots,A_n\in\mathcal{A}\). We can write \(p_n(A_1,\ldots,A_n)\) in the following way
    \begin{equation}\label{momentos-booleanos}
        p_n(A_1,\ldots,A_n)=h_n(A_1,\ldots,A_n)+\sum_{1\leq i<j\leq n}b^h_{n-j+i+1}(A_j,A_{j+1},\ldots,A_i)h_{j-i-1}(A_{i+1},\ldots,A_{j-1}),
    \end{equation}
    where \(b_n^h(A_1,\ldots,A_n)\) represents the \(n\)-th Boolean cumulant with respect to \(h\).
\end{lemma}

\begin{proof}
    Let \(A_1,\ldots,A_n\in\mathcal{A}\). Using the decomposition of moments in terms of Boolean cumulants, the $n$-th moment of $h$ can be written as 
    \begin{align*}
        h_n(A_1,\ldots,A_n)=&\sum_{\pi\in\mathcal{I}(n)}b_\pi^h[A_1,\ldots,A_n]\\ 
        =&\sum_{l=1}^n\sum_{1=i(1)<\cdots<i(l)\leq n}\quad \prod_{k=1}^lb_{i(k+1)-i(k)}^h(A_{i(k)},\ldots,A_{i(k+1)-1}).
    \end{align*}
    Combining this identity with \eqref{momentos_tau_cumulantes_booleanos_mu} from Lemma~\ref{lema-mom-cum-booleanos}, we obtain
    \begin{align*}
        p_n(A_1,\ldots,A_n)=&h_n(A_1,\ldots,A_n)+\left(p_n(A_1,\ldots,A_n)-h_n(A_1,\ldots,A_n)\right)\\
        =&h_n(A_1,\ldots,A_n)+\sum_{l=1}^n\sum_{1<i(1)<\cdots<i(l)\leq n}\quad \prod_{k=1}^lb_{i(k+1)-i(k)}^h(A_{i(k)},\ldots,A_{i(k+1)-1}).
    \end{align*}
    Moreover
    \begin{align*}
           \sum_{l=1}^n&\sum_{1<i(1)<\cdots<i(l)\leq n}\quad \prod_{k=1}^lb_{i(k+1)-i(k)}^h(A_{i(k)},\ldots,A_{i(k+1)-1})\\ 
           =&\sum_{l=1}^n\sum_{1<i(1)<i(l)\leq n}\quad \prod_{k=1}^{l-1}b_{i(k+1)-i(k)}^h(A_{i(k)},\ldots,A_{i(k+1)-1})b_{n-i(l)+i(1)}^h(A_{i(l)},\ldots,A_{i(1)-1})\\ 
           =&\sum_{l=1}^n\sum_{1<i(1)<i(l)\leq n}h_{i(l)-i(1)}(A_{i(1)},\ldots,A_{i(l)-1})b_{n-i(l)+i(1)}^h(A_{i(l)},\ldots,A_{i(1)-1}).
    \end{align*}
    By making a change of variable \(i(l)=j\) and \(i(1)=i+1\), we obtain the following
    \begin{align*}
        \sum_{l=1}^n&\sum_{1<i(1)<\cdots<i(l)\leq n}\quad \prod_{k=1}^lb_{i(k+1)-i(k)}^h(A_{i(k)},\ldots,A_{i(k+1)-1})\\ 
        =&\sum_{1\leq i<j\leq n }h_{j-i-1}(A_{i+1},\ldots,A_{j-1})b_{n-j+i+1}^h(A_{j},\ldots,A_{i}).
    \end{align*}
    Then
    \begin{equation*}
        p_n(A_1,\ldots,A_n)=h_n(A_1,\ldots,A_n)+\sum_{1\leq i<j\leq n }b_{n-j+i+1}^h(A_{j},\ldots,A_{i})h_{j-i-1}(A_{i+1},\ldots,A_{j-1}).
    \end{equation*}
\end{proof}
Following identity \eqref{momentos-booleanos}, and making use of the formula expressing boolean cumulants in terms of moments, one obtains
\begin{equation*}
    p_n(A_1,\ldots,A_n)=h_n(A_1,\ldots,A_n)+\sum_{\stackbin[]{}{1\leq i<j\leq n}}b^h_{n-j+i+1}(A_j,A_{j+1},\ldots,A_i)h_{j-i-1}(A_{i+1},\ldots,A_{j-1})
\end{equation*}
\begin{equation}
    \label{momentos-mixtos}
    =h_n(A_1,\ldots,A_n)+\sum_{\stackbin[]{}{1\leq i<j\leq n}} \left[\sum_{\sigma\in\mathcal{I}[j,\ldots,i]}(-1)^{\#(\sigma)-1}h_{\sigma}[A_1,\ldots,A_n]\right]h_{j-i-1}(A_{i+1},\ldots,A_{j-1}),
\end{equation}
where \(\mathcal{I}[j,\ldots,i]\) represents the set of interval partitions in which each partition has \(j\) as its first element and \(i\) as its last element. In order to prove Theorem \ref{moment-moment}, the goal is to rewrite the double sum in \eqref{momentos-mixtos} as a sum over the cyclic interval partitions of $[n]$, excluding the one-block partition $\{1,2,\ldots,n\}$. More precisely, we aim to transform 
\begin{equation*}
        \sum_{\stackbin[]{}{1\leq i<j\leq n}} \left[\sum_{\sigma\in \mathcal{I}[j,\ldots,i]}(-1)^{\#(\sigma)-1}h_{\sigma}[A_1,\ldots,A_n]\right]h_{j-i-1}(A_{i+1},\ldots,A_{j-1})
    \end{equation*}
    to
    \begin{equation*}
        \sum_{\stackbin[\pi\neq\{1,2,\ldots,n\}]{}{\pi\in \mathcal{CI}(n)}}h_\pi[A_1,\ldots,A_n](-1)^{\#(\pi)-1}.
    \end{equation*}
Lets consider $\xi:D\times \mathcal{I}[j,\ldots,i]\longrightarrow\mathcal{CI}(n)\backslash\{1,2,\ldots,n\}$, where  $D=\{(i,j)\in\mathbb{Z}^2:1\leq i<j\leq n\}$, defined by
   \begin{equation}\label{aplicación}
       \xi\left((i,j),\sigma\right)=\sigma\cup\{i+1,\dots,j-1\}.
   \end{equation}
If $\xi$ were bijective, it would be possible to perform the change directly; however, this is not the case. The following theorem formalizes this result
\begin{teorema}\label{moment-moment}
    Let $A_1,\ldots,A_n\in \mathcal{A}$ random variables. The multivariate moments $p_n(A_1,\ldots,A_n)$ can be represented by
    \begin{equation*}
    p_n(A_1,\ldots,A_n)=\sum_{\pi\in\mathcal{CI}(n)}(-1)^{\#(\pi)-1}h_\pi[A_1,\ldots,A_n],
    \end{equation*}
    where $\mathcal{CI}(n)$ is the set of cyclic intervals partitions of $n$.
\end{teorema}
\begin{proof}
Let $\pi\in\mathcal{CI}(n)\backslash\{1,2,\ldots,n\}$ be a cyclic interval partition with $r$ intervals, which we write in the following form
\begin{equation*}
        \pi=\{V_1,V_2,\ldots,V_r\}=\left\{\{i_1,i_1+1,\ldots,i_2-1\},\{i_2,i_2+1,\ldots,i_3-1\},\ldots,\{i_{r},i_{r}+1\ldots,i_1-1\}\right\}.
    \end{equation*}
 Observe that we have two cases: \textbf{Case 1)} some block may contain the points $1$ and $n$; \textbf{Case 2)} they may belong to different blocks. See Figure \ref{fig:caso general fig 1} and \ref{fig:caso general fig2}.
\begin{center}
    \begin{figure}[ht]
    \centering

    \begin{minipage}[t]{0.45\textwidth}
        \centering
        \begin{tikzpicture}
    \def\r{2}
    \def\R{2.5}

    \draw[thick] (0,0) circle (\r);

    \foreach \k/\ang/\lab in {
        1/90/{$i_{r-1}$},
        2/65.45/{$i_r-1$},
        3/40.91/{ },
        4/16.36/{$n$},
        5/-8.18/{$1$},
        6/-32.73/{},
        7/-57.27/{$i_1$},
        8/-81.82/{$i_2-1$},
        9/-106.36/{$i_2$},
        10/-130.91/{$i_3-1$},
        11/-155.45/{ },
        12/-180/{ }
    }{
        \coordinate (P\k) at ({\r*cos(\ang)},{\r*sin(\ang)});
        \coordinate (Q\k) at ({\R*cos(\ang)},{\R*sin(\ang)});

        \fill (P\k) circle (1.2pt);

\ifnum\k=1
    \node[font=\scriptsize, yshift=3pt] 
    at ({0.5*\r*cos(\ang)},{0.85*\r*sin(\ang)}) {\lab};
\else
    \ifnum\k=8
        \node[font=\scriptsize, yshift=-3pt] 
        at ({0.8*\r*cos(\ang)},{0.8*\r*sin(\ang)}) {\lab};
    \else
        \node[font=\scriptsize] 
        at ({0.8*\r*cos(\ang)},{0.8*\r*sin(\ang)}) {\lab};
    \fi
\fi
        
        \draw[blue, thick] (P\k) -- (Q\k);
    }

    \draw[blue, thick] (Q1) arc[start angle=90, end angle=65.45, radius=\R];

    \draw[blue, thick] (Q3) arc[start angle=40.91, end angle=-32.73, radius=\R];

    \draw[blue, thick] (Q7) arc[start angle=-57.27, end angle=-81.82, radius=\R];

    \draw[blue, thick] (Q9) arc[start angle=-106.36, end angle=-130.91, radius=\R];

    \draw[blue, thick] (Q11) arc[start angle=-155.45, end angle=-180, radius=\R];

\draw[orange, thick]
    ({0.7*\r*cos((65.45+40.91)/2)},{0.7*\r*sin((65.45+40.91)/2)})
    --
    ({\R*cos((65.45+40.91)/2)},{\R*sin((65.45+40.91)/2)});

\draw[orange, thick]
    ({0.7*\r*cos((-32.73-57.27)/2)},{0.7*\r*sin((-32.73-57.27)/2)})
    --
    ({\R*cos((-32.73-57.27)/2)},{\R*sin((-32.73-57.27)/2)});

\end{tikzpicture}
        \caption{Graphical representation of the case where $1$ and $n$ belong to the same block. And we consider $i=i_r-1$ and $j=i_1$.}
        \label{fig:caso general fig 1}
    \end{minipage}
    \hfill
    \begin{minipage}[t]{0.45\textwidth}
        \centering
        \begin{tikzpicture}
    \def\r{2}
    \def\R{2.5}

    \draw[thick] (0,0) circle (\r);

    \foreach \k/\ang/\lab in {
        1/90/{$i_{r-1}$},
        2/65.45/{$i_r-1$},
        3/40.91/{ $i_r$},
        4/16.36/{$n$},
        5/-8.18/{$1$},
        6/-32.73/{$i_1-1$},
        7/-57.27/{$i_1$},
        8/-81.82/{$i_2-1$},
        9/-106.36/{$i_2$},
        10/-130.91/{$i_3-1$},
        11/-155.45/{ },
        12/-180/{ }
    }{
        \coordinate (P\k) at ({\r*cos(\ang)},{\r*sin(\ang)});
        \coordinate (Q\k) at ({\R*cos(\ang)},{\R*sin(\ang)});

        \fill (P\k) circle (1.2pt);

        \ifnum\k=1
            \node[font=\scriptsize, yshift=3pt] 
            at ({0.8*\r*cos(\ang)},{0.8*\r*sin(\ang)}) {\lab};
        \else
            \ifnum\k=8
                \node[font=\scriptsize, yshift=-3pt] 
                at ({0.8*\r*cos(\ang)},{0.8*\r*sin(\ang)}) {\lab};
            \else
                \node[font=\scriptsize] 
                at ({0.8*\r*cos(\ang)},{0.8*\r*sin(\ang)}) {\lab};
            \fi
        \fi

        \draw[blue, thick] (P\k) -- (Q\k);
    }

\draw[blue, thick] (Q1) arc[start angle=90, end angle=65.45, radius=\R];

\draw[blue, thick] (Q3) arc[start angle=40.91, end angle=16.36, radius=\R];
\draw[blue, thick] (Q5) arc[start angle=-8.18, end angle=-32.73, radius=\R];

\draw[blue, thick] (Q7) arc[start angle=-57.27, end angle=-81.82, radius=\R];
\draw[blue, thick] (Q9) arc[start angle=-106.36, end angle=-130.91, radius=\R];
\draw[blue, thick] (Q11) arc[start angle=-155.45, end angle=-180, radius=\R];

\draw[orange, thick]
    ({0.7*\r*cos((16.36-8.18)/2)},{0.7*\r*sin((16.36-8.18)/2)})
    --
    ({\R*cos((16.36-8.18)/2)},{\R*sin((16.36-8.18)/2)});


    \draw[blue, thick] (Q7) arc[start angle=-57.27, end angle=-81.82, radius=\R];
    \draw[blue, thick] (Q9) arc[start angle=-106.36, end angle=-130.91, radius=\R];
    \draw[blue, thick] (Q11) arc[start angle=-155.45, end angle=-180, radius=\R];

    \draw[orange, thick]
        ({0.7*\r*cos((16.36-8.18)/2)},{0.7*\r*sin((16.36-8.18)/2)})
        --
        ({\R*cos((16.36-8.18)/2)},{\R*sin((16.36-8.18)/2)});


\end{tikzpicture}
        \caption{Graphical representation of the case where $1$ and $n$ belong to different blocks. And we consider $i=n$ and $j=1$.}
        \label{fig:caso general fig2}
    \end{minipage}
\end{figure}
\end{center}
\begin{itemize}
    \item \textbf{Case $1$ and $n$ in the same block:} We consider $((i,j),\rho)\in D\times \mathcal{I}[j,\ldots,i]$, where the pair $(i,j)$ is chosen so that $i$ and $j$ are separated by exactly one block of $\pi$. Under this condition, we may consider the following interval partitions:
        \begin{align*}
            \rho_1=&\{V_2,V_3,\ldots,V_{r}\}\in\mathcal{I}[i_2,\ldots,i_1-1],\\ 
            \rho_2=&\{V_3,V_4,\ldots,V_1\}\in\mathcal{I}[i_3,\ldots,i_2-1],\\
            \rho_3=&\{V_4,V_5,\ldots,V_2\}\in\mathcal{I}[i_4,\ldots,i_{3}-1],\\ 
            \vdots&\\ 
            \rho_{r-1}=&\{V_r,V_1,\ldots,V_{r-2}\}\in\mathcal{I}[i_{r},\ldots,i_{r-1}-1].
        \end{align*}
        Observe that we do not consider $\rho_0=\{V_1,V_2,\ldots,V_{r-1}\}\in\mathcal{I}[i_1,\ldots,i_{r}-1]$ since in this case one would obtain $j\leq i$, which is not allowed, graphically, this situation is illustrated in Figure~\ref{fig:caso general fig 1}. For the above $r-1$ elements $((i,j),\rho)$, one has
        \begin{equation*}
            \xi((i_1-1,i_2),\rho_1)=\xi((i_2-1,i_3),\rho_2)=\cdots=\xi((i_{r-1}-1,i_{r}),\rho_{r-1})=\pi.
        \end{equation*}
        Then, for $((i_1-1,i_2),\rho_1)$ we have
        \begin{align*}
            (-1)^{\#(\rho_1)-1}h_{\rho_1}[A_1,\ldots,A_n]&h_{i_{2}-(i_1-1)-1}(A_{i},\ldots,A_{i_2-1})
            \\=&(-1)^{r-2}h_{\#(V_1)}(A_{i_1},\ldots,A_{i_2-1})\cdots h_{\#(V_r)}(A_{i_{r}},\ldots,A_{i_1-1})\\ 
            =&(-1)^{\#(\pi)-2}h_\pi[A_1,\ldots,A_n], 
        \end{align*}
        and the same result is obtained for $((i_{2}-1,i_3),\rho_2),\ldots,((i_{r-1}-1,i_r),\rho_{r-1})$.
        
        On the other hand, by considering $((i,j),\sigma)\in D\times \mathcal{I}[j,\ldots,i]$, where the pair $(i,j)$ is such that they are consecutive, we can consider 
        \begin{align*}
            \sigma_1=&\{V_1,V_2,\ldots,V_r\}\in\mathcal{I}[i_1,\ldots,i_1-1],\\
            \sigma_2=&\{V_2,V_3,\ldots,V_1\}\in\mathcal{I}[i_2,\ldots,i_2-1],\\ 
            \vdots&\\ 
            \sigma_{r-1}=&\{V_{r-1},V_r,\ldots,V_{r-2}\}\in\mathcal{I}[i_{r-1},\ldots,i_{r-1}-1],\\
            \sigma_r=&\{V_r,V_1,\ldots,V_{r-1}\}\in\mathcal{I}[i_{r},\ldots,i_r-1].
        \end{align*}
        For these $r$ elements $((i,j),\sigma)$, we have
        \begin{equation*}
            \xi((i_1-1,i_1),\sigma_1)=\xi((i_2-1,i_2),\sigma_2)=\cdots=\xi((i_{r}-1,i_{r}),\sigma_{r})=\pi.
        \end{equation*}
        Then, for $((i_1-1,i_1),\sigma_1)$ we have
        \begin{align*}
            (-1)^{\#(\sigma_1)-1}h_{\sigma_1}[A_1,\ldots,A_n]&h_{i_{1}-(i_1-1)-1}(A_{i_1},\ldots,A_{i_1-1})\\
            =&(-1)^{r-1}h_{\#(V_1)}(A_{i_1},\ldots,A_{i_2-1})\cdots h_{\#(V_r)}(A_{i_{r}},\ldots,A_{i_1-1})\\ 
            =&(-1)^{\#(\pi)-1}h_\pi[A_1,\ldots,A_n], 
        \end{align*}
        and the same result is obtained for $((i_{2}-1,i_2),\sigma_2),\ldots,((i_{r}-1,i_{r}),\sigma_{r})$. All of the above implies that, in this case, upon expanding the expression 
        \begin{equation*}
            \sum_{\stackbin[]{}{1\leq i<j\leq n}} \left[\sum_{\sigma\in \mathcal{I}[j,\ldots,i]}(-1)^{\#(\sigma)-1}h_{\sigma}[A_1,\ldots,A_n]\right]h_{j-i-1}(A_{i+1},\ldots,A_{j-1}),
        \end{equation*}
        the term $h_\pi[A_1,\ldots,A_n]$ appears a total of $2r-1$ times, of which $r-1$ occur with sign $(-1)^{\#(\pi)-2}$ and $r$ occur with sign $(-1)^{\#(\pi)-1}$. After simplification, we obtain
        \begin{equation*}
            (r-1)(-1)^{\#(\pi)-2}h_\pi[A_1,\ldots,A_n]+r(-1)^{\#(\pi)-1}h_\pi[A_1,\ldots,A_n]=(-1)^{\#(\pi)-1}h_\pi[A_1,\ldots,A_n].
        \end{equation*}
        
    \item \textbf{Case $1$ and $n$ in different blocks:} We consider $((i,j),\rho)\in D\times \mathcal{I}[j,\ldots,i]$, where the pair $(i,j)$ is chosen so that $i$ and $j$ are separated by exactly one block of $\pi$. Under this condition, we may consider the following
        \begin{align*}
            \rho_2=&\{V_3,V_4,\ldots,V_1\}\in\mathcal{I}[i_3,\ldots,i_2-1],\\
            \rho_3=&\{V_4,V_5,\ldots,V_2\}\in\mathcal{I}[i_4,\ldots,i_{3}-1],\\ 
            \vdots&\\ 
            \rho_{r-1}=&\{V_r,V_1,\ldots,V_{r-2}\}\in\mathcal{I}[i_{r},\ldots,i_{r-1}-1]. 
        \end{align*}
    Note that we do not consider $\rho_0=\{V_1,V_2,\ldots,V_{r-1}\}\in\mathcal{I}[1,\ldots,i_{r}-1]$ or $\rho_1=\{V_2,V_3,\ldots,V_{r}\}\in\mathcal{I}[i_2,\ldots,n]$, since in both cases we would obtain $j\leq i$, which is not valid. Graphically, this can be seen in Figure~\ref{fig:caso general fig2}. For the above $r-2$ elements $((i,j),\rho)$, one has
    \begin{equation*}
        \xi((i_2-1,i_3),\rho_2)= \xi((i_3-1,i_4),\rho_3)=\cdots=\xi((i_{r-1}-1,i_{r}),\rho_{r-1})=\pi.
    \end{equation*}
    Then, for $((i_2-1,i_3),\rho_2)$ we have
        \begin{align*}
            (-1)^{|\rho_2|-1}h_{\rho_2}[A_1,\ldots,A_n]&h_{i_{3}-(i_2-1)-1}(A_{i_2},\ldots,A_{i_3-1})
            \\=&(-1)^{r-2}h_{|V_1|}(A_{i_1},\ldots,A_{i_2-1})\cdots h_{|V_r|}(A_{i_{r}},\ldots,A_{i_1-1})\\ 
            =&(-1)^{|\pi|-2}h_\pi[A_1,\ldots,A_n], 
        \end{align*}
        and the same result is obtained for each $((i_{3}-1,i_4),\rho_3),\ldots,((i_{r-1}-1,i_r),\rho_{r-1})$.
        
        On the other hand, by considering $((i,j),\sigma)\in D\times \mathcal{I}[j,\ldots,i]$, where the pair $(i,j)$ is such that they are consecutive, we can consider 
        \begin{align*}
            \sigma_2=&\{V_2,V_3,\ldots,V_1\}\in\mathcal{I}[i_2,\ldots,i_2-1],\\ 
            \sigma_3=&\{V_3,V_4,\ldots,V_2\}\in\mathcal{I}[i_3,\ldots,i_3-1],\\ 
            \vdots&\\ 
            \sigma_{r-1}=&\{V_{r-1},V_r,\ldots,V_{r-2}\}\in\mathcal{I}[i_{r-1},\ldots,i_{r-1}-1],\\
            \sigma_r=&\{V_r,V_1,\ldots,V_{r-1}\}\in\mathcal{I}[i_{r},\ldots,i_r-1],
        \end{align*}
        observe that in this case we do not consider $\sigma_1=\{V_1,V_2,\ldots,V_r\}\in \mathcal{I}[1,\ldots,n]$, since in this case one would obtain $n\leq 1$, which is false. Graphically, this can be seen in Figure~\ref{fig:caso general fig2}. For the above $r-1$ elements $((i,j),\sigma)$, one has
        \begin{equation*}
            \xi((i_2-1,i_2),\sigma_2)=\xi((i_3-1,i_3),\sigma_3)=\cdots=\xi((i_{r}-1,i_{r}),\sigma_{r})=\pi.
        \end{equation*}
        Then, for $((i_2-1,i_2)$ we have
        \begin{align*}
            (-1)^{\#(\sigma_2)-1}h_{\sigma_2}[A_1,\ldots,A_n]&h_{i_{1}-(i_1-1)-1}(A_{i_2},\ldots,A_{i_2-1})\\
            =&(-1)^{r-1}h_{\#(V_1)}(A_{i_1},\ldots,A_{i_2-1})\cdots h_{\#(V_r)}(A_{i_{r}},\ldots,A_{i_1-1})\\ 
            =&(-1)^{\#(\pi)-1}h_\pi[A_1,\ldots,A_n], 
        \end{align*}
         and the same result is obtained for each $((i_{3}-1,i_3),\sigma_3),\ldots,((i_{r}-1,i_{r}),\sigma_{r})$. Therefore, all of the above implies that, in this case, upon expanding the expression 
         \begin{equation*}
            \sum_{\stackbin[]{}{1\leq i<j\leq n}} \left[\sum_{\sigma\in \mathcal{I}[j,\ldots,i]}(-1)^{\#(\sigma)-1}h_{\sigma}[A_1,\ldots,A_n]\right]h_{j-i-1}(A_{i+1},\ldots,A_{j-1}),
        \end{equation*}
        the term $h_\pi[A_1,\ldots,A_n]$ appears a total of $2r-3$ times, of which $r-2$ occur with sign $(-1)^{\#(\pi)-2}$ and $r-1$ occur with sign $(-1)^{\#(\pi)-1}$. After simplification, we obtain
        \begin{equation*}
            (r-2)(-1)^{\#(\pi)-2}h_\pi[A_1,\ldots,A_n]+(r-1)(-1)^{\#(\pi)-1}h_\pi[A_1,\ldots,A_n]=(-1)^{\#(\pi)-1}h_\pi[A_1,\ldots,A_n].
        \end{equation*}
    
\end{itemize}
In both cases, whether $1$ and $n$ belong to the same block or to different blocks, we conclude that for $\pi\in\mathcal{CI}(n)\backslash\{1,2,\ldots,n\}$, when expanding the expression 
\begin{equation*}
    \sum_{\stackbin[]{}{1\leq i<j\leq n}} \left[\sum_{\sigma\in \mathcal{I}[j,\ldots,i]}(-1)^{\#(\sigma)-1}h_{\sigma}[A_1,\ldots,A_n]\right]h_{j-i-1}(A_{i+1},\ldots,A_{j-1}),
\end{equation*} 
the term $(-1)^{\#(\pi)-1}h_\pi(A_1,\ldots,A_n)$ appears. This allows us to obtain the following equality
\begin{equation*}
    \sum_{\stackbin[]{}{1\leq i<j\leq n}} \left[\sum_{\sigma\in \mathcal{I}[j,\ldots,i]}(-1)^{\#(\sigma)-1}h_{\sigma}[A_1,\ldots,A_n]\right]h_{j-i-1}(A_{i+1},\ldots,A_{j-1})= \sum_{\stackbin[\pi\neq\{1,2,\ldots,n\}]{}{\pi\in \mathcal{CI}(n)}}h_\pi[A_1,\ldots,A_n](-1)^{\#(\pi)-1}. 
\end{equation*}
Finally, this concludes the proof
\begin{align*}
    p_n(A_1,\ldots,A_n)=&h_n(A_1,\ldots,A_n)+ \sum_{\stackbin[\pi\neq\{1,2,\ldots,n\}]{}{\pi\in \mathcal{CI}(n)}}h_\pi[A_1,\ldots,A_n](-1)^{\#(\pi)-1}\\
    =& \sum_{\stackbin[]{}{\pi\in \mathcal{CI}(n)}}h_\pi[A_1,\ldots,A_n](-1)^{\#(\pi)-1}.
\end{align*}
\end{proof}
\begin{remark}[Relation with infinitesimal freeness and Boolean-like cumulants]
If we impose traciality on the functional $h$, then for every $\{1_{n,j}\}=\{(j,j+1,\ldots,j-1)\}\in \mathcal{CI}(n)$ we have
\[
h_{1_{n,j}}[A_1,\ldots,A_n] = h_n(A_1,\ldots,A_n), \qquad j=2,3,\ldots,n.
\]
Consequently, following the Remark \ref{remark-cyclic-partitions}, Theorem~\ref{moment-moment} reduces to
\begin{equation}\label{moment-moment-tracial}
    p_n(A_1,\ldots,A_n)
    =(n-1)h_n(A_1,\ldots,A_n)
    +\sum_{\pi\in \widehat{\mathcal{CI}}(n)}(-1)^{\#(\pi)-1}h_\pi[A_1,\ldots,A_n].
\end{equation}
Somewhat surprisingly, this expression is closely related to the Boolean-like cumulants introduced by J. Mingo and P.-L. Tseng in the context of infinitesimal probability \cite{mingo2024infinitesimal}. If the reader would like to know more about infinitesimal probability we recommend \cite{belinschi2012free,fevrier2010infinitesimal,
shlyakhtenko2018free,
mingo2019non}. In this particular work, the authors consider an infinitesimal probability space $(\mathcal{A},\varphi,\varphi^\prime)$, an infinitesimal idempotent element $j\in\mathcal{A}$ ($j=j^2$) with $\varphi^\prime(j)\neq 0$, and a unital subalgebra $\mathcal{B}\subset\mathcal{A}$ such that $\{j\}$ and $\mathcal{B}$ are infinitesimally free. The authors define Boolean-like cumulants using $\widehat{\mathcal{CI}}(n)$ by
\begin{equation}\label{boolean-like-definition}
    \widetilde{b}_n(A_1,\ldots,A_n)
    :=(n-1)\varphi(A_1\cdots A_n)
    +\sum_{\pi\in\widehat{\mathcal{CI}}(n)}(-1)^{\#(\pi)-1}\varphi_\pi[A_1,\ldots,A_n].
\end{equation}
where $A_1,\ldots,A_n\in \mathcal{B}$. Comparing \eqref{moment-moment-tracial} and \eqref{boolean-like-definition}, it follows that $\widetilde{b}_n = p_n$ once $\widetilde{b}_n$ is identified with the moments of the inverse Markov--Krein transform and $\varphi$ with the moments of the Markov--Krein transform. In fact, under the same assumptions, Proposition~5.3 in \cite{mingo2024infinitesimal} shows that
\begin{equation}\label{boolean-like}
\varphi^\prime(j)\widetilde{b}_n(A_1,\dots,A_n)
=\varphi^\prime(A_1\cdots A_n)
-\varphi^\prime\big((1-j) A_1 (1-j) A_2 (1-j) \cdots (1-j) A_n\big).
\end{equation}
In a concrete setting, let $\mathcal{H}$ be a separable Hilbert space, $\mathcal{A}\subset B(\mathcal{H})$, and $\xi\in\mathcal{H}$ a unit vector. Then, in analogy with the above framework, one may consider the identifications 
\begin{align*}
    & j \longleftrightarrow p, \\
    & 1-j \longleftrightarrow q, \\
    & \varphi(\cdot) \longleftrightarrow \langle \cdot\,\xi,\xi\rangle, \\
    & \varphi^\prime(\cdot) \longleftrightarrow \operatorname{Tr}(\cdot),
\end{align*}
where $p$ is the rank-one projection onto $\mathrm{span}\{\xi\}$ and $q=1-p$. In this case, $\varphi^\prime(j)=\operatorname{Tr}(p)=1$, and \eqref{boolean-like} yields
\[
\widetilde{b}_n(A_1,\ldots,A_n)
=\operatorname{Tr}(A_1\cdots A_n)
-\operatorname{Tr}(qA_1 qA_2 q \cdots qA_n).
\]
Which coincide with the multivariate moments of the inverse Markov--Krein transform. Therefore $\widetilde{b}_n=p_n$. On the other hand, in the one-dimensional case, let $(\mathcal{A},\varphi,\varphi^\prime)$ be an infinitesimal probability space and let $a,p\in\mathcal{A}$ be infinitesimally free, with $p$ an infinitesimal idempotent and $q=1-p$. Define
\[
G_\tau(z)=\sum_{n=0}^\infty \frac{p_n}{z^{n+1}}, 
\qquad
G_\mu(z)=\sum_{n=0}^\infty \frac{h_n}{z^{n+1}},
\]
where $p_n=\varphi^\prime(a^n)-\varphi^\prime((qaq)^n)$ and $h_n=\varphi(a^n)$. Then the relation
\(G_\mu^\prime(z)+G_\tau(z)G_\mu(z)=0
\)
holds.
\end{remark}
By the previous theorem, the first multivariate moments $p_n(A_1,\ldots,A_n)$ can be expressed in terms of the functionals $h_n$ as follows:
\begin{align*}
    p_1(\mathcolor{blue}{A})
    &=
    h_1(\mathcolor{blue}{A}),
    \\
    p_2(\mathcolor{blue}{A},\mathcolor{orange}{B})
    &=
    h_2(\mathcolor{blue}{A},\mathcolor{orange}{B})
    +
    h_2(\mathcolor{orange}{B},\mathcolor{blue}{A})
    -
    h_1(\mathcolor{blue}{A})h_1(\mathcolor{orange}{B}),
    \\
    p_3(\mathcolor{blue}{A},\mathcolor{orange}{B},\mathcolor{LimeGreen}{C})
    &=
    h_3(\mathcolor{blue}{A},\mathcolor{orange}{B},\mathcolor{LimeGreen}{C})
    +
    h_3(\mathcolor{orange}{B},\mathcolor{LimeGreen}{C},\mathcolor{blue}{A})
    +
    h_3(\mathcolor{LimeGreen}{C},\mathcolor{blue}{A},\mathcolor{orange}{B}) -
    h_1(\mathcolor{blue}{A})h_2(\mathcolor{orange}{B},\mathcolor{LimeGreen}{C})
    \\
    &\quad
    -
    h_1(\mathcolor{orange}{B})h_2(\mathcolor{LimeGreen}{C},\mathcolor{blue}{A})
    -
    h_1(\mathcolor{LimeGreen}{C})h_2(\mathcolor{blue}{A},\mathcolor{orange}{B})+
    h_1(\mathcolor{blue}{A})
    h_1(\mathcolor{orange}{B})
    h_1(\mathcolor{LimeGreen}{C}),
    \\
    p_4(\mathcolor{blue}{A},\mathcolor{orange}{B},\mathcolor{LimeGreen}{C},D)
    &=
    h_4(\mathcolor{blue}{A},\mathcolor{orange}{B},\mathcolor{LimeGreen}{C},D)
    +
    h_4(\mathcolor{orange}{B},\mathcolor{LimeGreen}{C},D,\mathcolor{blue}{A})
    +
    h_4(\mathcolor{LimeGreen}{C},D,\mathcolor{blue}{A},\mathcolor{orange}{B})
    +
    h_4(D,\mathcolor{blue}{A},\mathcolor{orange}{B},\mathcolor{LimeGreen}{C})
    \\
    &\quad
    -
    h_1(\mathcolor{blue}{A})h_3(\mathcolor{orange}{B},\mathcolor{LimeGreen}{C},D)
    -
    h_1(\mathcolor{orange}{B})h_3(\mathcolor{LimeGreen}{C},D,\mathcolor{blue}{A})-
    h_1(\mathcolor{LimeGreen}{C})h_3(D,\mathcolor{blue}{A},\mathcolor{orange}{B})
    \\
    &\quad
    -
    h_1(D)h_3(\mathcolor{blue}{A},\mathcolor{orange}{B},\mathcolor{LimeGreen}{C})-
    h_2(\mathcolor{blue}{A},\mathcolor{orange}{B})
    h_2(\mathcolor{LimeGreen}{C},D)
    -
    h_2(\mathcolor{orange}{B},\mathcolor{LimeGreen}{C})
    h_2(D,\mathcolor{blue}{A})
    \\
    &\quad
    +
    h_1(\mathcolor{blue}{A})h_1(\mathcolor{orange}{B})h_2(\mathcolor{LimeGreen}{C},D)
    +
    h_1(\mathcolor{blue}{A})h_1(D)h_2(\mathcolor{orange}{B},\mathcolor{LimeGreen}{C})
    \\
    &\quad
    +
    h_1(\mathcolor{LimeGreen}{C})h_1(D)h_2(\mathcolor{blue}{A},\mathcolor{orange}{B})
    +
    h_1(\mathcolor{orange}{B})h_1(\mathcolor{LimeGreen}{C})h_2(D,\mathcolor{blue}{A})
    \\
    &\quad
    -
    h_1(\mathcolor{blue}{A})
    h_1(\mathcolor{orange}{B})
    h_1(\mathcolor{LimeGreen}{C})
    h_1(D).
\end{align*}
\begin{remark}
Inverting the formula from Theorem~\ref{moment-moment} is, in principle, not possible, since the number of terms appearing in \(p_n\) and \(h_n\) is different. This issue is resolved under the additional assumption that \(h\) is tracial. In this case, the formula from Theorem~\ref{moment-moment} can be inverted; however, no explicit closed formula is currently known.
\end{remark}
\begin{corollary}
Let $A_1, \dots, A_n \in \mathcal{A}$ be random variables. If  $h$ is tracial, the multivariate moments $h_n(A_1, \dots, A_n)$ can be represented by
\begin{equation*}
    h_n(A_1, \dots, A_n) = \sum_{\pi \in \mathcal{NC}(n)} c(\pi) p_\pi[A_1, \dots, A_n],    
\end{equation*}
where $ c: \mathcal{NC}(n) \rightarrow \mathbb{C} $.
\end{corollary}
\begin{dem}
The result follows by a straightforward induction argument.    
\end{dem}

Since the multifunctionals $p_n$ and $h_n$ extend the univariate case, Theorem \ref{moment-moment} can be written in the univariate case as
\begin{equation}\label{momentos-momentos-univariados}
p_n
=
\sum_{\pi \in \mathcal{CI}(n)}
(-1)^{\#(\pi)-1}
h_\pi.
\end{equation}
This identity provides a new formula in the univariate setting, expressing the moments of the inverse Markov--Krein transform, $p_n$, in terms of the moments of the Markov--Krein transform, $h_n$. In particular, the first moments can be written as
\begin{equation*}
\begin{aligned}
p_1 &= h_1,
\\
p_2 &= 2h_2 - h_1^2,
\\
p_3 &= 3h_3 - 3h_2 h_1 + h_1^3,
\\
p_4 &= 4h_4 - 4h_3 h_1 - 2h_2^2 + 4h_2 h_1^2 - h_1^4.
\end{aligned}
\end{equation*}
\begin{corollary}\label{p_n en relación con h_n}
    In the one-dimensional Markov-Krein correspondence the following are equivalents
    \begin{enumerate}
        \item[(1) ] $\displaystyle p_n=\sum_{\pi\in\mathcal{P}(n)}(-1)^{\#(\pi)-1}\frac{\#(\pi)!(\#(\pi)-1)!}{(n-1)!}h_\pi$.
        \item[(2) ] $\displaystyle  p_n=\sum_{\pi\in\mathcal{CI}(n)}(-1)^{\#(\pi)-1}h_\pi$.
    \end{enumerate}
\end{corollary}
In view of \eqref{momentos-momentos-univariados} and Examples \ref{ejemplo-1} and \ref{ejemplo-2}, it follows that
\begin{equation*}
    m_n^{\mathrm{bern}}
    =
    \sum_{\pi \in \mathcal{CI}(n)}
    (-1)^{\#(\pi)-1}
    m^{\mathrm{arc}}_\pi
    \qquad \text{and} \qquad
    m_n^{\mathrm{arc}}
    =
    \sum_{\pi \in \mathcal{CI}(n)}
    (-1)^{\#(\pi)-1}
    m^{\mathrm{sc}}_\pi,
\end{equation*}
where \( m_n^{\mathrm{bern}}, m_n^{\mathrm{arc}}, \) and \( m_n^{\mathrm{sc}} \) denote the \(n\)-th moments of the Bernoulli\((-2,2)\), arcsine\((-2,2)\), and semicircular\((-2,2)\) distributions, respectively. Substituting the second identity into the first, we obtain
\begin{align*}
    m_n^{\mathrm{bern}}
    = \sum_{\pi \in \mathcal{CI}(n)}
    (-1)^{\#(\pi)-1}
    m^{\mathrm{arc}}_\pi =& \sum_{\pi \in \mathcal{CI}(n)}
    (-1)^{\#(\pi)-1}
    \prod_{V\in \pi} m_{|V|}^{\mathrm{arc}} \\
    =& \sum_{\pi \in \mathcal{CI}(n)}
    (-1)^{\#(\pi)-1}
    \prod_{V\in \pi}
    \left[
        \sum_{\sigma_V\in \mathcal{CI}(|V|)}
        (-1)^{\#(\sigma_V)-1}
        m^{\mathrm{sc}}_{\sigma_V}
    \right].
\end{align*}
In particular, this provides an explicit formula expressing the moments of the Bernoulli\((-2,2)\) distribution in terms of those of the semicircular\((-2,2)\) distribution:
\begin{equation*}
    m_n^{\mathrm{bern}}
    =-
    \sum_{\pi \in \mathcal{CI}(n)}
    \prod_{V\in \pi}
    \left[
        \sum_{\sigma_V\in \mathcal{CI}(|V|)} (-1)^{\#(\sigma_V)}
        m^{\mathrm{sc}}_{\sigma_V}
    \right].
\end{equation*}
More generally, whenever the Markov--Krein transform can be iterated twice, that is, when the Rayleigh measure \(\text{MK}^{-1}(\mu)\) is a probability measure, there exists a unique Rayleigh measure \(\text{MK}^{-1}(\text{MK}^{-1}(\mu))\) such that
\begin{equation*}
    p_n^{\text{MK}^{(-1)}(\text{MK}^{(-1)}(\mu))}
    =-
    \sum_{\pi \in \mathcal{CI}(n)}
    \prod_{V\in \pi}
    \left[
        \sum_{\sigma_V\in \mathcal{CI}(|V|)} (-1)^{\#(\sigma_V)}
        h^{\text{MK}(\tau)}_{\sigma_V}
    \right],
\end{equation*}
where
\begin{equation*} 
    G_{\text{MK}^{(-1)}(\text{MK}^{(-1)}(\mu))}(z)=\sum_{n=0}^\infty\frac{1}{z^{n+1}}p_n^{\text{MK}^{(-1)}(\text{MK}^{(-1)}(\mu))}
    \qquad \text{and}\qquad G_\mu(z)=\sum_{n=0}^\infty\frac{1}{z^{n+1}}h^{\text{MK}(\tau)}_n,
\end{equation*}
are the Cauchy transforms satisfying the following identities 
\begin{equation}\label{ecs-cauchy}
G_{\text{MK}^{(-1)}}(z)=-\frac{G_\mu^{\prime}(z)}{G_\mu(z)},\qquad 
    G_{\text{MK}^{(-1)}(\text{MK}^{(-1)}(\mu))}(z)=\frac{G_\mu^\prime(z)}{G_\mu(z)}-\frac{G_\mu^{\prime\prime}(z)}{G_\mu^\prime(z)}.
\end{equation}
\begin{remark}[Connection with finite free probability and infinitesimal distributions]
Another somewhat surprising connection of the Markov--Krein transform arises in finite free probability and limit for roots of polynomials.
We briefly recall their definition and refer the reader to \cite{arizmendi2023finite,arizmendi2018cumulants,arizmendi2025finite,marcus2022finite}. Let
$
p_d(x)=\sum_{i=0}^d x^{d-i}(-1)^i a_i^{p_d},
$
be a monic polynomial of degree $d$ with real roots and consider its empirical root distribution
\[
\mu_{p_d}=\frac{1}{d}\sum_{i=1}^d \delta_{\lambda_i(p_d)},
\]
where $\lambda_1(p_d),\ldots,\lambda_d(p_d)$ are the roots of $p_d$.  The $n$-th moment of $p_d$ is defined as the $n$-th moment of $\mu_{p_d}$, namely
\[
m_n(p_d)=\frac{1}{d}\sum_{i=1}^d \lambda_i^n(p_d).
\]
The finite free cumulants of $p_d$ were introduced in \cite{arizmendi2018cumulants} and are the sequence $(k_n^{(d)}(p_d))_{n=1}^d$, defined in terms of the coefficients as follows, 
\[
k_n^{(d)}(p_d)=\frac{(-d)^{n-1}}{(n-1)!}\sum_{\pi\in\mathcal{P}(n)} \widetilde{a}_\pi^{p_d}\, \mu(0_n,\pi),
\]
where $\widetilde{a}_i^{p_d}=\binom{d}{i}^{-1} a_i^{p_d}$. In the large $d$ limit,  these cumulants satisfy that $ k_n^{(d)}(p_d) \to r_n(\mu),$ 
exactly when $ m_n(p_d) \to m_n(\mu).$ 

In the same framework, we consider a sequence of real-rooted polynomials $(p_d)_{d=1}^\infty$ with asymptotic root distribution $\mu$. Let $m_n(\mu)=\int t^n\,\mu(dt)$ and define
\[
m_n^\prime=\lim_{d\to\infty} \big(m_n(p_d)-m_n(\mu)\big),
\]
whenever the limit exists. Denoting by $\mu^\prime$ the signed measure with moments $(m_n^\prime)_{n\geq 1}$, we refer to $\mu^\prime$ as the infinitesimal distribution associated with $(p_d)_{d=1}^\infty$, and to the pair $(\mu,\mu^\prime)$ as its infinitesimal asymptotic distribution. 
\begin{theorem}[Theorem 1.6 in \cite{arizmendi2023finite}]
Let $\mu$ be a compactly supported probability measure on $\mathbb{R}$. Suppose there exists a sequence of monic real-rooted polynomials $(p_d)_{d=1}^\infty$, where $p_d$ has degree $d$, such that $k_n^d(p_d)=r_n(\mu)$ for all $n \leq d$. Then $\mu_{p_d}$ admits an infinitesimal asymptotic distribution $(\mu,\mu^\prime)$, whose infinitesimal Cauchy transform is given by
\[
G_{\mu^\prime}(z)
= \frac{G_\mu^{\prime\prime}(z)}{2G_\mu^\prime(z)}
- \frac{G_\mu^\prime(z)}{G_\mu(z)}.
\]
\end{theorem}
In particular, when $\text{MK}^{-1}(\mu)$  is again a probability measure,
\[
\mu^\prime
= \frac{1}{2}\big(
\text{MK}^{-1}(\mu)
- \text{MK}^{-1}(\text{MK}^{-1}(\mu))
\big).
\]
Consequently, the infinitesimal moments are explicitly given by
\begin{align*}
m_n^\prime
&= \frac{1}{2}\left[
p_n^{\text{MK}^{-1}(\mu)}
- p_n^{\text{MK}^{-1}(\text{MK}^{-1}(\mu))}
\right] \\
&= \frac{1}{2}\left[
\sum_{\pi \in \mathcal{CI}(n)} (-1)^{\#(\pi)-1}
\, h^{\text{MK}(\tau)}_{\pi}
+
\sum_{\pi \in \mathcal{CI}(n)}
\prod_{V\in \pi}
\left(
\sum_{\sigma_V\in \mathcal{CI}(|V|)}(-1)^{\#(\sigma_V)}
h^{\text{MK}(\tau)}_{\sigma_V}
\right)
\right].
\end{align*}
\end{remark}

\subsection{Moment-cumulant formulas}
In this section, we present a unified description of multivariate moments, $(p_n)_{n\geq 1}$ in terms of the classical, boolean and free cumulants of $h$.
\subsubsection{Boolean moment-cumulant formula}
The next theorem establishes a new formula for $p_n$ in terms of the Boolean cumulants with respect to $h$.
\begin{theorem}
    Let $A_1,\ldots,A_n\in \mathcal{A}$ be random variables. Then
    \begin{equation*}
        p_n(A_1,\ldots,A_n)=\sum_{\pi\in\mathcal{CI}(n)}b_\pi^h[A_1,\ldots,A_n],
    \end{equation*}
    where $b_n^h$ represent the $n-$th Boolean cumulant with respect to $h$.
\end{theorem}
\begin{proof}
    The proof is straightforward using Lemma \ref{lema-mom-cum-booleanos}, which states that
\begin{equation}\label{thm-boolean}
    p_n(A_1,\ldots,A_n)=\sum_{l=1}^n \sum_{1\leq i(1)<\cdots<i(l)\leq n} \prod_{k=1}^l b_{i(k+1)-i(k)}^h(A_{i(k)},A_{i(k)+1},\ldots,A_{i(k+1)-1}).
\end{equation}
Observe that each partition $\pi\in\mathcal{CI}(n)$ determines unique indices
$i(1),i(2),\ldots,i(s)$ such that
\begin{equation*}
    b_\pi^h[A_1,\ldots,A_n]=b^h(A_{i(1)},\ldots,A_{i(2)-1})b^h(A_{i(2)},\ldots,A_{i(3)-1})\cdots b^h(A_{i(s)},\ldots,A_{i(1)-1}),
\end{equation*}
and conversely, each collection of indices $\{i(1),\ldots,i(s)\}$ defines a unique partition $\pi\in\mathcal{CI}(n)$,
\[
\pi=\{\{i(1),\ldots,i(2)-1\},\{i(2),\ldots,i(3)-1\},\ldots,\{i(s),\ldots,i(1)-1\}\},
\]
such that
\begin{equation*}
    b_\pi^h[A_1,\ldots,A_n]=b^h(A_{i(1)},\ldots,A_{i(2)-1})b^h(A_{i(2)},\ldots,A_{i(3)-1})\cdots b^h(A_{i(s)},\ldots,A_{i(1)-1}).
\end{equation*}
This observation allows us to rewrite the relation \eqref{thm-boolean} as
\begin{equation*}
    p_n(A_1,\ldots,A_n)=\sum_{\pi\in\mathcal{CI}(n)}b_\pi^h[A_1,\ldots,A_n].
\end{equation*}
\end{proof}
\begin{remark}
    In the one-dimensional setting, if $\mu\in \mathcal{M}(\mathbb{R})$ and $\tau\in\mathcal{F}(\mathbb{R})$ are related through the Markov--Krein correspondence, then
    \begin{equation*}
        p_n(\tau)
        =
        \sum_{\pi\in \mathcal{CI}(n)}b_\pi^\mu,
    \end{equation*}
    where $b_n^\mu$ denotes the $n$-th Boolean cumulant of $\mu$. For
    \(
        \mu=\frac12\delta_{-2}+\frac12\delta_2,
    \)
    the Boolean cumulants are given by $b_1^\mu=0$, $b_2^\mu=4$, and $b_n^\mu=0$ for all $n\geq 3$. Hence, only pairing partitions contribute, and therefore
    \begin{equation*}
        p_{2n}(\tau)
        =
        \sum_{\pi\in \mathcal{CI}_2(2n)}
        \prod_{V\in\pi}b_{|V|}^\mu
        =4^n+4^n=
        2\cdot 4^n
        =
        2^{2n+1},
        \qquad n\geq 1,
    \end{equation*}
    and,
    \begin{equation*}
        p_{2n+1}(\tau)=0,
        \qquad n\geq 0.
    \end{equation*}
    These are precisely the moments of the signed measure
    \(
      \tau=2\mu-\delta_0,
    \)
    in agreement with Example~\ref{example-3}.
\end{remark}

\subsubsection{Free moment-cumulant formula}
The following result was originally proved in the univariate setting by T.~Hasebe and K.~Fujie \cite{katsunori}, in connection with the asymptotic behavior of spectra of principal submatrices of rotationally invariant random matrices. Related multivariate constructions have subsequently appeared in different areas of noncommutative probability, such as cyclic conditional freenes \cite{arizmendi} and free probability of type $B$ prime \cite{katsunori2023}. We present here a different proof based on combinatorial arguments.
\begin{theorem}\label{main-thm-cumulanteslibres}
    Let $A_1,\ldots,A_n\in \mathcal{A}$ be random variables. Then
    \begin{equation*}
        p_n(A_1,\ldots,A_n)=\sum_{\pi\in\mathcal{NC}(n)}(n+1-\#(\pi))r_\pi[A_1,\ldots,A_n],
    \end{equation*}
    where $r_n$ represent the $n-$th free cumulant with respect to $h$.
\end{theorem}
\begin{proof}
    Using formula \eqref{43}, we can express Boolean cumulants in terms of free cumulants via irreducible partitions. That is, if $r_n^h$ denotes the $n$-th free cumulant with respect to $h$, then
    \begin{equation*}
        b_n^h(A_1,\ldots,A_n)=\sum_{\pi\in \mathcal{NC}_{irr}(n)}r_\pi^h[A_1,\ldots,A_n],
    \end{equation*}
     where $\mathcal{NC}_{\mathrm{irr}}(n)$ denotes the set of irreducible non-crossing partitions of $[n]$.
If we consider $l \in \{1,2,\ldots,n\}$ and $1 \leq i(1) < \cdots < i(l) \leq n$, with the convention that $i(n+1) = i(1)$, then
    \begin{align*}
        b_{i(k+1)-i(k)}^h(A_1,\ldots,A_n)=&\sum_{\stackbin[]{}{\pi_k\in\mathcal{NC}_{irr}(i(k),i(k)+1,\ldots,i(k+1)-1)}}r^h_{\pi_k}[A_{i(k)},A_{i(i)+1},\ldots,A_{i(k+1)-1}]\\ 
        =&\sum_{\stackbin[i(k+1)-1\sim_{\pi_k} i(k)]{}{\pi_k\in\mathcal{NC}(i(k),i(k)+1,\ldots,i(k+1)-1)}}r^h_{\pi_k}[A_{i(k)},A_{i(i)+1},\ldots,A_{i(k+1)-1}],
    \end{align*}
    for all $k=1,2,\ldots,l$. Substituting the previous equivalence into expression \eqref{momentos_tau_cumulantes_booleanos_mu}, we obtain
    \begin{align*}
       p_n&(A_1,\ldots,A_n)\\
       =&\sum_{l=1}^n\sum_{1\leq i(1)<\cdots<i(l)\leq n}\quad \underbrace{\prod_{k=1}^l\quad 
        \sum_{\stackbin[i(k+1)-1\sim_{\pi_k} i(k)]{}{\pi_k\in\mathcal{NC}((i(k),i(k)+1,\ldots,i(k+1)-1))}}r^h_{\pi_k}[A_{i(k)},A_{i(i)+1},\ldots,A_{i(k+1)-1}]}_{(\star)}.
    \end{align*}
    Note that, by the way the partitions $\pi_k$ are chosen, they are pairwise disjoint. Consider $\pi=\pi_1\cup\cdots\cup\pi_l$ and denote by
\begin{equation*}
    C:=\{\pi:i(2)-1\sim_{\pi_1} i(1),\, i(3)-1\sim_{\pi_2} i(2),\, \ldots,\, i(1)-1\sim_{\pi_l} i(l)\},
\end{equation*}
the set on which the events $\{i(k+1)-1\sim_{\pi_k} i(k)\}$ occur for each $k=1,2,\ldots,l$. Observe, in addition, that the set $C$ can be represented as
\begin{equation*}
    C=\{\pi:\{i(2)-1,i(3)-1,\ldots,i(n)-1,i(1)-1\}\in Kr(\pi)\},
\end{equation*}
where $Kr(\pi)$ denotes the Kreweras complement of $\pi$. Then the term $(\star)$ can be represented as follows:
\begin{align*}
    (\star)=&\prod_{k=1}^l
    \sum_{\stackbin[i(k+1)-1\sim_{\pi_k} i(k)]{}{\pi_k\in\mathcal{NC}((i(k),i(k)+1,\ldots,i(k+1)-1))}}
    r^h_{\pi_k}[A_{i(k)},A_{i(k)+1},\ldots,A_{i(k+1)-1}]\\
    =&\left[\sum_{\stackbin[i(2)-1\sim_{\pi_1}i(1)]{}{\pi_1\in\mathcal{NC}(i(1),\ldots,i(2)-1)}}r^h_{\pi_1}[A_{i(1)},\ldots,A_{i(2)-1}]\right]\cdots
    \left[\sum_{\stackbin[i(1)-1\sim_{\pi_l}i(l)]{}{\pi_l\in\mathcal{NC}(i(l),\ldots,i(1)-1)}}r^h_{\pi_l}[A_{i(l)},\ldots,A_{i(1)-1}]\right] \\
    =&\sum_{\stackbin[\pi\in C]{}{\pi\in\mathcal{NC}(n)}}r^h_\pi[A_1,A_2,\ldots,A_n],
\end{align*}
which implies that
\begin{align*}
     p_n(A_1,\ldots,A_n)
    =&\sum_{l=1}^n\sum_{1\leq i(1)<\cdots<i(l)\leq n}
    \sum_{\stackbin[\pi\in C]{}{\pi\in\mathcal{NC}(n)}}r^h_\pi[A_1,\ldots,A_n]\\
    =&\sum_{\stackbin[]{}{\pi\in\mathcal{NC}(n)}}\sum_{l=1}^n\sum_{\stackbin[\pi\in C]{1\leq i(1)<\cdots<i(l)\leq n}{}}
    r^h_\pi[A_1,\ldots,A_n]\\
    =&\sum_{\pi\in\mathcal{NC}(n)}\sum_{\sigma\in Kr(\pi)}r^h_\pi[A_1,\ldots,A_n].
\end{align*}
Finally, since $|Kr(\pi)|=n+1-\#(\pi)$, we conclude that
\begin{equation*}
    p_n(A_1,\ldots,A_n)=\sum_{\pi\in\mathcal{NC}(n)}(n+1-\#(\pi))\,r^h_\pi[A_1,\ldots,A_n].
\end{equation*}
\end{proof}
\subsubsection{Classical moment-cumulant formula}
In the following result, we present a new formula expressing $p_n$ in therms to the classical cumulant with respect to $h$.
\begin{theorem}\label{main-thm-cumulantesclasicos}
    Let $A_1,\ldots,A_n\in \mathcal{A}$ be random variables. Then
    \begin{equation*}
        p_n(A_1,\ldots,A_n)=\sum_{\pi\in\mathcal{P}(n)}(n+1-|\overline{\pi}|)c_\pi^h[A_1,\ldots,A_n],
    \end{equation*}
    where $c_n^h$ represent the $n-$th classical cumulant with respect to $h$.
\end{theorem}
\begin{proof}
    We know from the Theorem \ref{main-thm-cumulanteslibres} that the moments of the inverse of the Markov--Krein transform can be written in terms of the free cumulants of the Markov-Krein transform as follows:
\begin{equation}\label{23}
    p_n(A_1,\ldots,A_n)=\sum_{\pi\in\mathcal{NC}(n)}(n+1-\#(\pi))r_\pi^h[A_1,\ldots,A_n].
\end{equation}
Using the formula expressing free cumulants in terms of classical cumulants \eqref{41}, we have
\begin{equation*}
      r_n^h(A_1,\ldots,A_n)=\sum_{\stackbin[]{}{\sigma\in\mathcal{P}_{conn}(n)}}c_\sigma^h[A_1,\ldots,A_n],
\end{equation*}
which can be extended to a family of multilinear functionals $r_\pi^h$ ($\pi\in\mathcal{P}(n)$, $n\geq 1$) as follows:
\begin{equation*}
    r_\pi^h[A_1,\ldots,A_n]=\sum_{\stackbin[\overline{\sigma}=\pi]{}{\sigma\in\mathcal{P}(n)}}c_\sigma^h[A_1,\ldots,A_n].
\end{equation*}
Substituting the previous expression into \eqref{23}, we obtain
\begin{align*}
    p_n(A_1,\ldots,A_n)=&\sum_{\pi\in\mathcal{NC}(n)}(n+1-\#(\pi))\sum_{\stackbin[\overline{\sigma}=\pi]{\sigma\in\mathcal{P}(n)}{}}c_\sigma^h[A_1,\ldots,A_n]\\ 
    =&\sum_{\pi\in\mathcal{NC}(n)}\sum_{\stackbin[\overline{\sigma}=\pi]{}{\sigma\in\mathcal{P}(n)}}(n+1-\#(\pi))c_\sigma^h[A_1,\ldots,A_n]\\ 
    =&\sum_{\pi\in\mathcal{NC}(n)}\sum_{\stackbin[\overline{\sigma}=\pi]{}{\sigma\in\mathcal{P}(n)}}(n+1-\#(\overline{\sigma}))c_\sigma^h[A_1,\ldots,A_n].
\end{align*}
Finally, note that for $\sigma\in\mathcal{P}(n)$ there is only one partition $\pi\in\mathcal{NC}(n)$ such that $\overline{\sigma}=\pi$. Then 
\begin{equation*}
    p_n(A_1,\ldots,A_n)=\sum_{\sigma\in\mathcal{P}(n)}(n+1-\#(\overline{\sigma}))c_\sigma^h[A_1,\ldots,A_n].
\end{equation*}
\end{proof}
\begin{remark}
     In \cite{interlacingmeasures}, Kerov proved that if $\mu\sim N(0,1)$ and $\tau$ denotes the inverse of the Markov--Krein transform, then $\tau$ is the probability measure with density 
    \begin{equation*}
        \frac{d\mu(t)}{dt}=\frac{1}{\sqrt{2\pi}}\frac{1}{|D_{-1}(it)|^2},\qquad -\infty<t<\infty,
    \end{equation*}
    where $D_{-1}(z)=e^{z^2/2}\int_z^\infty e^{-x^2/2}dx,$
    is a parabolic cylinder function. Since the classical cumulants of $\mu\sim N(0,1)$ are \((c_1,c_2,c_3,\ldots)=(0,1,0,\ldots)
    \), Theorem~\ref{main-thm-cumulantesclasicos} provides a explicit formula for the moments of $\tau$
    \begin{equation*}
        p_{2n}(\tau)=\sum_{\pi\in\mathcal{P}_2(2n)}(2n+1-\#(\overline{\pi})).
    \end{equation*}
    The first few moments are given by the numbers
    \begin{equation*}
        p_0(\tau)=1,\quad p_2(\tau)=2,\quad p_4(\tau)=10,\quad p_6(\tau)=74,\quad p_8(\tau)=706,\quad p_{10}(\tau)=8162.
    \end{equation*}
    These numbers have appeared in several combinatorial contexts since the 1950s. For example, Touchard \cite{touchard1952probleme} interpreted $p_n$ as the number of indecomposable systems of $(n+1)$ arcs. These numbers also appear in the combinatorics of continued fractions and Motzkin-type paths, as discussed by Kerov in \cite{interlacingmeasures}.
    
\end{remark}

\subsection*{Future research}
\begin{itemize}
    \item[(1)] According to the formulas obtained in this work, it is natural
    to investigate whether one can establish a formula expressing \(p_n\) in
    terms of the monotone cumulants with respect to \(h\). After several
    attempts to derive such a formula, we observed that the underlying
    combinatorial structure becomes particularly interesting, especially due
    to a certain cyclic behavior appearing in monotone partitions and their
    associated nesting trees.
    \item[(2)] There exists a generalization of the Markov--Krein transform in
which a probability measure \(\mu\in\mathcal{M}(\mathbb{R})\) and a
probability measure \(\nu\in\mathcal{M}(\mathbb{R})\), satisfying
\(
\int_{\mathbb{R}}\log(1+|t|)\,\nu(dt)<\infty,
\)
are linked through the generalized Markov--Krein transform with parameter
\(c>0\) by the identity
\[
\int_{-\infty}^{\infty}\frac{\mu(dt)}{(z-t)^c}
=
\exp\left(
-c\int_{-\infty}^{\infty}\log(z-u)\,\nu(du)
\right).
\]
The measures \(\mu\) and \(\nu\) uniquely determine each other; see
\cite{lijoi2004means}. This generalization appears in several contexts
\cite{benaych2022matrix,mergny2022rank}. The question is whether there exists a multivariate version
of the Markov--Krein transform in this more general setting.
\end{itemize}
\subsection*{Acknowledgments} 
\noindent
Most of the results presented in this work were obtained during the author's master's thesis at CIMAT under the supervision of Octavio Arizmendi and Josué Vázquez. The author is deeply grateful to them for their guidance, support, valuable discussions, and for taking the time to carefully read this manuscript.
\bibliography{ref}

@article{interlacingmeasures,
  title={Interlacing measures},
  author={Kerov, Sergei},
  journal={Translations of the American Mathematical Society-Series 2},
  volume={181},
  pages={35--84},
  year={1997},
  publisher={Providence [etc.] American Mathematical Society, 1949-}
}

@article{katsunori,
  author  = {Fujie, Katsunori and Hasebe, Takahiro},
  title   = {The Spectra of Principal Submatrices in Rotationally Invariant Hermitian Random Matrices and the Markov--Krein Correspondence},
  journal = {ALEA: Latin American Journal of Probability and Mathematical Statistics},
  volume  = {19},
  number  = {1},
  pages   = {109--123},
  year    = {2022},
  doi     = {10.30757/ALEA.v19-05},
  issn    = {1980-0436},
  mrnumber = {4359787}
}

@article{flajolet1980combinatorial,
  title={Combinatorial aspects of continued fractions},
  author={Flajolet, Philippe},
  journal={Discrete Mathematics},
  volume={32},
  number={2},
  pages={125--161},
  year={1980},
  publisher={Elsevier}
}

@article{lifshits1952problem,
  title={On a problem of the theory of perturbations connected with quantum statistics},
  author={Lifshits, Il'ya Mikhailovich},
  journal={Uspekhi matematicheskikh nauk},
  volume={7},
  number={1},
  pages={171--180},
  year={1952},
  publisher={Russian Academy of Sciences, Steklov Mathematical Institute of Russian~…}
}

@article{yongdiagrams-kerov,
  title={A q-analog of the hook walk algorithm and random Young tableaux},
  author={Kerov, S},
  journal={Functional Analysis and Its Applications},
  volume={26},
  number={3},
  pages={179--187},
  year={1992},
  publisher={Springer}
}

@book{harary2014graphical,
  title={Graphical enumeration},
  author={Harary, Frank and Palmer, Edgar M},
  year={2014},
  publisher={Elsevier}
}

@book{OnExponential,
  title={On exponential representations of analytic functions in the upper half-plane with positive imaginary part},
  author={Aronszajn, Nachman and Donoghue, William F},
  volume={20},
  pages={321-388},
  year={1957},
  publisher={University of Kansas, Department of Mathematics}
}

@book{kreuin1977markov,
  title={The Markov moment problem and extremal problems},
  author={Krein, Mark Grigorévich and Nudel, AA and Louvish, D}, 
  year={1977},
  publisher={American Mathematical Society}
}

@article{katsunori2023,
  author  = {Fujie, Katsunori and Hasebe, Takahiro},
  title   = {Free Probability of Type B Prime},
  journal = {Transactions of the American Mathematical Society},
  volume  = {378},
  number  = {8},
  pages   = {5551--5577},
  year    = {2025},
  doi     = {10.1090/tran/9464}
}

@article{arizmendi2023cyclic,
  title={Cyclic independence: Boolean and monotone},
  author={Arizmendi, Octavio and Hasebe, Takahiro and Lehner, Franz},
  journal={Algebraic Combinatorics},
  volume={6},
  number={6},
  pages={1697--1734},
  year={2023}
}

@article{diaz2022combinatorics,
  title={Combinatorics of NC-probability spaces with independent constants},
  author={Diaz-Aguilera, Carlos and Gaxiola, Tulio and Santos, Jorge and Vargas, Carlos},
  journal={Infinite Dimensional Analysis, Quantum Probability and Related Topics},
  volume={25},
  number={02},
  pages={2250009},
  year={2022},
  publisher={World Scientific}
}

@article{mingo2024infinitesimal,
  title={Infinitesimal operators and the distribution of anticommutators and commutators},
  author={Mingo, James A and Tseng, Pei-Lun},
  journal={Journal of Functional Analysis},
  volume={287},
  number={9},
  pages={110591},
  year={2024},
  publisher={Elsevier}
}

@article{arizmendi,
  title={Combinatorics of cyclic-conditional freeness},
  author={Arizmendi, Octavio and C{\'e}bron, Guillaume and Gilliers, Nicolas},
  journal={arXiv preprint arXiv:2311.13178},
  year={2023}
}

@article{marcus2022finite,
  title={Finite free convolutions of polynomials},
  author={Marcus, Adam W and Spielman, Daniel A and Srivastava, Nikhil},
  journal={Probability Theory and Related Fields},
  volume={182},
  number={3},
  pages={807--848},
  year={2022},
  publisher={Springer}
}

@article{arizmendi2025finite,
  title={Finite Free Convolution: Infinitesimal Distributions},
  author={Arizmendi, Octavio and Perales, Daniel and Vazquez-Becerra, Josue},
  journal={arXiv preprint arXiv:2505.01705},
  year={2025}
}

@book{mingo2017free,
  title={Free probability and random matrices},
  author={Mingo, James A and Speicher, Roland},
  volume={35},
  year={2017},
  publisher={Springer}
}

@article{touchard1952probleme,
  title={Sur un probleme de configurations et sur les fractions continues},
  author={Touchard, Jacques},
  journal={Canadian Journal of Mathematics},
  volume={4},
  pages={2--25},
  year={1952},
  publisher={Cambridge University Press}
}

@article{arizmendi2018cumulants,
  title={Cumulants for finite free convolution},
  author={Arizmendi, Octavio and Perales, Daniel},
  journal={Journal of Combinatorial Theory, Series A},
  volume={155},
  pages={244--266},
  year={2018},
  publisher={Elsevier}
}

@article{arizmendi2023finite,
  title={Finite free cumulants: Multiplicative convolutions, genus expansion and infinitesimal distributions},
  author={Arizmendi, Octavio and Garza-Vargas, Jorge and Perales, Daniel},
  journal={Transactions of the American Mathematical Society},
  volume={376},
  number={06},
  pages={4383--4420},
  year={2023}
}

@article{arizmendi2016free,
  title={Free subordination and Belinschi--Nica semigroup},
  author={Arizmendi, Octavio and Hasebe, Takahiro},
  journal={Complex analysis and operator theory},
  volume={10},
  pages={581--603},
  year={2016},
  publisher={Springer}
}

@incollection{Krein-teoria-operadores,
  title={On certain new studies in the perturbation theory for selfadjoint operators},
  author={Krein, MG},
  booktitle={Topics in Differential and Integral Equations and Operator Theory},
  pages={107--172},
  year={1983},
  publisher={Springer}
}

@article{speicher1997boolean,
  title={Boolean convolution, Free Probability Theory, Proceedings, Toronto, Canada, 1995},
  author={Speicher, Roland},
  journal={Fields Inst Commun.},
  volume={12},
  pages={267--279},
  year={1997},
  publisher={Amer. Math. Soc.}
}

@article{speicher1994multiplicative,
  title={Multiplicative functions on the lattice of non-crossing partitions and free convolution},
  author={Speicher, Roland},
  journal={Mathematische Annalen},
  volume={298},
  number={1},
  pages={611--628},
  year={1994},
  publisher={Springer}
}

@incollection{rota1964foundations,
  title={On the foundations of combinatorial theory: I. Theory of M{\"o}bius functions},
  author={Rota, Gian-Carlo},
  booktitle={Classic Papers in Combinatorics},
  pages={332--360},
  year={1964},
  publisher={Springer}
}

@article{Schutzenberger1947,
  author  = {Sch{\"u}tzenberger, Marcel-Paul},
  title   = {Sur certains param{\`e}tres caract{\'e}ristiques des syst{\`e}mes d'{\'e}v{\'e}nements compatibles et d{\'e}pendants et leur application au calcul des cumulants de la r{\'e}p{\'e}tition},
  journal = {C. R. Acad. Sci. Paris},
  volume  = {225},
  year    = {1947},
  pages   = {277--278}
}

@book{aigner2012combinatorial,
  title={Combinatorial theory},
  author={Aigner, Martin},
  volume={234},
  year={2012},
  publisher={Springer Science \& Business Media}
}

@article{stanley2011enumerative,
  title={Enumerative Combinatorics Volume 1 second edition},
  author={Stanley, Richard P},
  journal={Cambridge studies in advanced mathematics},
  year={2011}
}

@book{book:nica,
  title={Lectures on the combinatorics of free probability},
  author={Nica, Alexandru and Speicher, Roland},
  volume={13},
  year={2006},
  publisher={Cambridge University Press}
}

@article{lehner2002free,
  title={Free cumulants and enumeration of connected partitions},
  author={Lehner, Franz},
  journal={European Journal of Combinatorics},
  volume={23},
  number={8},
  pages={1025--1031},
  year={2002},
  publisher={Elsevier}
}

@article{arizmendi2015relations,
  title={Relations between cumulants in noncommutative probability},
  author={Arizmendi, Octavio and Hasebe, Takahiro and Lehner, Franz and Vargas, Carlos},
  journal={Advances in Mathematics},
  volume={282},
  pages={56--92},
  year={2015},
  publisher={Elsevier}
}

@article{belinschi2008eta,
  title={$\eta$-series and a Boolean Bercovici--Pata bijection for bounded k-tuples},
  author={Belinschi, Serban T and Nica, Alexandru},
  journal={Advances in Mathematics},
  volume={217},
  number={1},
  pages={1--41},
  year={2008},
  publisher={Elsevier}
}

@book{shohat1943problem,
  title={The problem of moments},
  author={Shohat, James Alexander and Tamarkin, Jacob David},
  number={1},
  year={1943},
  publisher={American Mathematical Soc.}
}

@article{fevrier2010infinitesimal,
  title={Infinitesimal non-crossing cumulants and free probability of type B},
  author={F{\'e}vrier, Maxime and Nica, Alexandru},
  journal={Journal of Functional Analysis},
  volume={258},
  number={9},
  pages={2983--3023},
  year={2010},
  publisher={Elsevier}
}

@article{belinschi2012free,
  title={Free probability of type B: analytic interpretation and applications},
  author={Belinschi, Serban T and Shlyakhtenko, Dimitri},
  journal={American Journal of Mathematics},
  volume={134},
  number={1},
  pages={193--234},
  year={2012},
  publisher={Johns Hopkins University Press}
}

@article{shlyakhtenko2018free,
  title={Free probability of type-B and asymptotics of finite-rank perturbations of random matrices},
  author={Shlyakhtenko, Dimitri},
  journal={Indiana University Mathematics Journal},
  volume={67},
  number={2},
  pages={971--991},
  year={2018},
  publisher={JSTOR}
}

@article{mingo2019non,
  title={Non-crossing annular pairings and the infinitesimal distribution of the goe},
  author={Mingo, James A},
  journal={Journal of the London Mathematical Society},
  volume={100},
  number={3},
  pages={987--1012},
  year={2019},
  publisher={Wiley Online Library}
}

@article{mergny2022rank,
  title={Rank one HCIZ at high temperature: interpolating between classical and free convolutions},
  author={Mergny, Pierre and Potters, Marc},
  journal={SciPost Physics},
  volume={12},
  number={1},
  pages={022},
  year={2022}
}

@article{lijoi2004means,
  title={Means of a Dirichlet process and multiple hypergeometric functions},
  author={Lijoi, Antonio and Regazzini, Eugenio},
  year={2004}
}

@article{benaych2022matrix,
  title={Matrix addition and the Dunkl transform at high temperature},
  author={Benaych-Georges, Florent and Cuenca, Cesar and Gorin, Vadim},
  journal={Communications in Mathematical Physics},
  volume={394},
  number={2},
  pages={735--795},
  year={2022},
  publisher={Springer}
}
\bibliographystyle{alpha}
\end{document}